\documentclass[11pt,a4paper]{amsart}

\usepackage{amsmath}
\usepackage{amsfonts}
\usepackage{graphicx}
\usepackage{framed}
\usepackage{amssymb}
\usepackage{esvect}
\usepackage{xcolor}

\usepackage{cite}

\usepackage[backref=page]{hyperref}

\renewcommand*{\backref}[1]{}
\renewcommand*{\backrefalt}[4]{%
    \ifcase #1 (Not cited.)%
    \or        (Cited on page~#2.)%
    \else      (Cited on pages~#2.)%
    \fi}

\makeatletter
\@namedef{subjclassname@2020}{\textup{2020} Mathematics Subject Classification}
\makeatother

\usepackage{etex}
\usepackage{latexsym}
\usepackage[english]{babel}
\usepackage[utf8x]{inputenc}

\def \N {\mathbb{N}}
\def \R {\mathbb{R}}

\def \e {\varepsilon}

\def \LL {\mathcal{L}}

\def \d {\mathrm{d}}

\usepackage{amsthm}
\theoremstyle{definition}
\newtheorem{definition}{Definition}[section]

\newtheorem{remark}[definition]{Remark}

\theoremstyle{plain}

\newtheorem{theorem}[definition]{Theorem}

\newtheorem{lemma}[definition]{Lemma}

\renewcommand{\leq}{\leqslant}
\renewcommand{\geq}{\geqslant}
\renewcommand{\ge}{\geqslant}
\renewcommand{\le}{\leqslant}

\usepackage{mathtools}
\mathtoolsset{showonlyrefs}

\numberwithin{equation}{section}
\begin{document}

 \begin{abstract} We study a critical problem for an operator of mixed order obtained by the superposition of a Laplacian with a fractional Laplacian.

In particular, we investigate the corresponding Sobolev inequality, detecting the optimal constant, which we show that is never achieved.

Moreover, we present an existence (and nonexistence) theory for the corresponding subcritical perturbation problem.
 \end{abstract}
 
 \title[A Brezis-Nirenberg result]{A Brezis-Nirenberg type result \\ for mixed local and nonlocal operators}
 
 \author[S.\,Biagi]{Stefano Biagi}
 \author[S.\,Dipierro]{Serena Dipierro}
 \author[E.\,Valdinoci]{Enrico Valdinoci}
 \author[E.\,Vecchi]{Eugenio Vecchi}
 
 \address[S.\,Biagi]{Dipartimento di Matematica
 \newline\indent Politecnico di Milano \newline\indent
 Via Bonardi 9, 20133 Milano, Italy}
 \email{stefano.biagi@polimi.it}
 
 \address[S.\,Dipierro]{Department of Mathematics and Statistics
 \newline\indent University of Western Australia \newline\indent
 35 Stirling Highway, WA 6009 Crawley, Australia}
 \email{serena.dipierro@uwa.edu.au}
 
 \address[E.\,Valdinoci]{Department of Mathematics and Statistics
 \newline\indent University of Western Australia \newline\indent
 35 Stirling Highway, WA 6009 Crawley, Australia}
 \email{enrico.valdinoci@uwa.edu.au}
 
 \address[E.\,Vecchi]{Dipartimento di Matematica
 \newline\indent Università di Bologna \newline\indent
 Piazza di Porta San Donato 5, 40126 Bologna, Italy}
 \email{eugenio.vecchi2@unibo.it}



\thanks{The authors are members of INdAM.
S. Dipierro and E. Valdinoci are members of AustMS.
S. Dipierro is supported by
the Australian Research Council DECRA DE180100957
``PDEs, free boundaries and applications''.
E. Valdinoci is supported by the Australian Laureate Fellowship
FL190100081
``Minimal surfaces, free boundaries and partial differential equations''. S. Biagi and E. Vecchi are supported by the PRIN 2022 project 2022R537CS \emph{$NO^3$ - Nodal Optimization, NOnlinear elliptic equations, NOnlocal geometric problems, with a focus on regularity}, founded by the European Union - Next Generation EU.}

 \date{\today}
 
 \maketitle
 
\section{Introduction} In this paper we are concerned with
elliptic operators of mixed local and nonlocal type,
in relation to the possible existence of positive solutions for critical problems
and in connection with possible optimizers of suitable Sobolev inequalities.

The investigation of operators of mixed order is a very topical subject of investigation,
arising naturally in several fields, for instance as the superposition of different types of stochastic processes
such as a classical random walk and a L\'evy flight, which has also interesting application
in the study of optimal animal foraging strategies, see~\cite{MPV13, PV18, DLPV, DV21}.

{F}rom the technical point of view, these operators offer quite relevant challenges caused by the combination
of nonlocal difficulties with the lack of invariance under scaling. The contemporary investigation
has specifically focused on several problems in the existence and regularity theory 
(see~\cite{JK05, BI08, BJK10, BCCI12, CKSV12, BMV, CDV22, BDVV22a, SV, AC21, GaKi, GaLi, DeFMin, GaUk, SUPR})
symmetry and classification results
(see~\cite{BDVV, BVDV21, BDVV23}), etc.

The stirring motivation
for the problems presented in this paper comes from the study of nonlinear problems with critical exponents,
as the ones suggested by the optimizers of the Sobolev inequality.
Roughly speaking (see Section~\ref{sec.Prel} for a formal definition of the
functional setting) the strategy adopted here is to consider a fractional exponent~$s\in(0,1)$,
an open set~$\Omega\subseteq\R^n$,
not necessarily bounded or connected, and all functions~$u:\R^n\to\R$ which vanish outside~$\Omega$,
accounting for a mixed type Sobolev inequality
of the sort
\begin{equation}\label{T57tud}
\mathcal{S}_{n,s}(\Omega)\,\|u\|_{L^{2^*}(\R^n)}^2 
    \leq \|\nabla u\|_{L^2(\R^n)}^2+\iint_{\R^{2n}}\frac{|u(x)-u(y)|^2}{|x-y|^{n+2s}}\,dx\,dy.
\end{equation}
Here above, the constant~$\mathcal{S}_{n,s}(\Omega)$ is taken to be the largest 
possible one for which
such an inequality holds true and, as usual, $n\ge3$ and~$2^*:=\frac{2n}{n-2}$.

We observe below that indeed~\eqref{T57tud} is satisfied by choosing the constant on the left hand side
to be (less than or) equal to the classical Sobolev constant
\begin{equation} \label{eq:bestSnloc} 
 \mathcal{S}_n := \frac{1}{n(n-2)\pi}\,\bigg(\frac{\Gamma(n)}{\Gamma(n/2)}\bigg)^{2/n},
 \end{equation}
since by the standard Sobolev inequality (in~$\Omega$, or even in~$\R^n$),
$$ \mathcal{S}_{n}\,\|u\|_{L^{2^*}(\R^n)}^2 
    \leq \|\nabla u\|_{L^2(\R^n)}^2
    \leq \|\nabla u\|_{L^2(\R^n)}^2
    +\iint_{\R^{2n}}\frac{|u(x)-u(y)|^2}{|x-y|^{n+2s}}\,dx\,dy.
$$
As a result, the largest possible constant in~\eqref{T57tud} certainly satisfies~$\mathcal{S}_{n,s}(\Omega)\ge
\mathcal{S}_n$. In principle, one may suspect that in fact a strict inequality occurs
(because, for instance,~$\mathcal{S}_n$ is independent on~$\Omega$, as well as on~$s$), but this is not the case,
according to the following result:

 \begin{theorem} \label{thm:SnsSn}
  Let~$s\in(0,1)$ and $\Omega\subseteq\R^n$ be an arbitrary open set. Then, we have
  \begin{equation} \label{eq:SnsequalSn}
 	\mathcal{S}_{n,s}(\Omega) = \mathcal{S}_n.  
  \end{equation}
 \end{theorem}
 
 In our setting, Theorem~\ref{thm:SnsSn} plays a pivotal role, since
 it allows classical techniques to be efficiently adapted to the setting of operators with mixed order.
 
 A natural question related to this problem is whether or not the optimal constant in~\eqref{T57tud}
 is achieved, i.e. whether or not a minimizer\footnote{Strictly speaking,
 to formally address the problem of the existence of this minimizer, one should carefully define the functional
 space to which this object may in principle belong. This will be done in details in Section~\ref{sec.Prel}.
 Suffices here to say that one can, roughly speaking, just consider the natural closure of the space of smooth
 functions compactly supported in~$\Omega$.
 
 Further clarifications about this point will be highlighted in Remark~\ref{rem:Snasymptotic}.} exists. The next result answers this question.
 
 \begin{theorem} \label{thm:Snneverachieved}
  Let $\Omega\subseteq\R^n$ be an arbitrary open set. Then, the optimal constant~$\mathcal{S}_{n,s}(\Omega)$ in~\eqref{T57tud}
  is never achieved.
 \end{theorem}

As customary, the study of possible optimizers for inequalities of Sobolev type is intimately connected with
the possible existence of nontrivial solutions for critical problems like
\begin{equation*}
  \begin{cases}
  \LL u =  u^{2^*-1} & \text{in $\Omega$}, \\
  u \gneqq 0 & \text{in $\Omega$}, \\
  u \equiv 0 & \text{in $\R^n\setminus\Omega$},
  \end{cases}
 \end{equation*}
 where $\LL$ denotes (here and throughout the paper) the mixed local-nonlocal operator
 $$\LL u = -\Delta u+(-\Delta)^s u.$$
As for the classical case (i.e. $\LL = -\Delta$), when $\Omega$ is star-shaped, the above problem 
does not admit positive solutions, see \cite[Theorem 1.3]{RSPoho} thanks to some proper variants of the Pohozaev identity. This fact naturally drive us to study critical problem of the form
 \begin{equation} \label{eq:mainPBBN}
  \begin{cases}
  \LL u = u^{2^*-1}+\lambda u^p & \text{in $\Omega$}, \\
  u \gneqq 0 & \text{in $\Omega$}, \\
  u \equiv 0 & \text{in $\R^n\setminus\Omega$},
  \end{cases}
 \end{equation}
 where~$p\in[1,2^*-1)$ and~$\lambda\in\R$.
 We remind that when, $\LL=-\Delta$, problem
 \eqref{eq:mainPBBN} with $p = 1$
 is the famous Brezis-Nirenberg problem \cite{BrNir}, and it has already been studied for the fractional Laplacian as well (see~\cite{SerValTAMS}). 
 The case of nonlinear operators of fractional type has also been considered (see~\cite{MPSY16}).
 
 Our first result is that there do not exist
 solutions to this problem when $\lambda\leq 0$, at least on bounded and star-shaped domains. 
 
 \begin{theorem} \label{thm:nonexistence} Let~$\lambda\leq 0$.
  Assume that $\Omega\subseteq\R^n$ is bounded and star-shaped. 
  Then, there do not exist solutions to problem \eqref{eq:mainPBBN},
  whatever the exponent $p\in[1,2^*-1)$.
 \end{theorem}
To further analyze the existence theory for problem \eqref{eq:mainPBBN} in dependence of the parameter~$\lambda$, let us briefly recall the main strategy used in \cite{BrNir} in the case $p=1$. 

Due to the lack of compactness caused by the critical exponent, an idea borrowed from \cite{Aubin} consists in proving that
$$S(\lambda) := \inf \left\{  \|\nabla u\|^2_{L^2(\Omega)}- \lambda\|u\|^2_{L^2(\Omega)}:\,
\text{$u\in H^{1}_0(\Omega)$ and $\|u\|_{L^{2^*}(\Omega)} = 1$} \right\}$$
is achieved under some restrictions on $\lambda$; in order to do this, the key
step is to show that $S(\lambda) < \mathcal{S}_n$, where $\mathcal{S}_n$ is as in 
\eqref{eq:bestSnloc}.

Now, the \emph{strict inequality} $S(\lambda) < \mathcal{S}_n$ is obtained in \cite{BrNir}
via the following approach: first of all, taking into account that the
minimizers in the Sobolev inequality are given by Aubin-Talenti functions, one considers
the function
\begin{equation*}
u_\varepsilon = \dfrac{\phi}{(\varepsilon^2 + |x|^2)^{(n-2)/2}} \quad (\varepsilon>0).
\end{equation*} Then, using $u_\e$ as a \emph{competitor function}, one gets (at least for $n\geq 5$) that
$$\frac{\|\nabla u_\e\|^2_{L^2(\Omega)} - \lambda \|u_\e\|^2_{L^2(\Omega)}}
{\|u_\e\|^2_{L^{2^*}(\Omega)}} = \mathcal{S}_n - c\lambda\e^2+ O(\e^{n-2})\quad
\text{as $\e\to 0^+$},$$
where $c > 0$ is a suitable constant. From this, choosing $\e$ sufficiently small,
one immediately conclude that $S(\lambda) < \mathcal{S}_n$.
 A similar approach works for $1<p<2^* -1$ as well, and it 
 has also been used in the nonlocal framework \cite{SerValTAMS}.
 \medskip
 
 Differently from what one can expect, in our mixed setting the situation changes. The main reason is that, in  
 trying to repeat the above argument, one is led to consider the following
 minimization problem
 \begin{equation*}
 \begin{split}
 S(\lambda)
 & := \inf \left\{ \|\nabla u\|^2_{L^2(\Omega)}+[u]^2_s- \lambda\|u\|^2_{L^2(\Omega)}:\,
\text{$u\in H^{1}_0(\Omega)$ and $\|u\|_{L^{2^*}(\Omega)} = 1$} \right\}
 \end{split}
 \end{equation*}
 and to prove that
 \begin{equation} \label{eq:goalnonloc}
  S(\lambda) < \mathcal{S}_{n,s}(\Omega),
  \end{equation}
 where 
 $[u]_s$ denotes the usual Gagliardo seminorm of $u$ (here, we agree to identify $u$ with its
 zero-extension out of $\Omega$) and
 $\mathcal{S}_{n,s}(\Omega)$ is as in \eqref{T57tud}.
 We explicitly stress that the constant $S(\lambda)$
 (which is different from the one defined above) actually depends
 on $n,s$ and on $\Omega$; however, since these quantities are fixed,
 to simplify the notation we avoid to keep track of these dependencies in the sequel.
 
 Now, since we know from Theorem \ref{thm:SnsSn} that $\mathcal{S}_{n,s}(\Omega) = \mathcal{S}_n$,
 in order to prove \eqref{eq:goalnonloc}
 it is natural to consider as a competitor function
 \emph{the same function} $u_\e$ defined above; however, the presence
 of the nonlocal part $[u_\e]^2_s$ gives
 \begin{align*}
  & \frac{\|\nabla u_\e\|^2_{L^2(\Omega)} + [u_\e]^2_s - \lambda \|u_\e\|^2_{L^2(\Omega)}}
{\|u_\e\|^2_{L^{2^*}(\Omega)}} \\[0.1cm]
& \qquad\qquad = \mathcal{S}_n + O(\e^{2-2s}) - c\lambda\e^2+ O(\e^{n-2})\quad
\text{as $\e\to 0^+$}
\end{align*}
 (see, precisely, identity \eqref{eq:tociteIntro} in the proof of Lemma \ref{lem:Path}),
 and the term $O(\e^{2-2s})$ \emph{is not negligible} when $\e\to 0^+$. 
  \medskip
  
 All that being said, in the linear case $p = 1$
 we obtain that problem~\eqref{eq:mainPBBN} does not admit any solution both in the range of ``small'' and ``large'' values of~$\lambda$, but it does possess solutions for an ``intermediate'' regime of 
 values of~$\lambda$; more precisely, denoting by 
 $\lambda_{1,s}$ the first Dirichlet eigenvalue of $(-\Delta)^s$ in a bounded open set~$\Omega$,
 and by~$\lambda_1$ the first Dirichlet eigenvalue of $\LL$ in $\Omega$ (a precise summary of the related spectral
property being recalled on page~\pageref{eq:deflambda1s}), we 
have the following result.
 \begin{theorem} \label{thm:mainLinear} 
Let~$\Omega\subset\R^n$ be an open and bounded set and $p=1$.
  There exists some $\lambda^*\in[\lambda_{1,s},\lambda_1)$ 
  such that pro\-blem \eqref{eq:mainPBBN} possesses at least one
  solution if 
  $$\lambda^*<\lambda<\lambda_1.$$
  Moreover, the following facts hold:
  \begin{enumerate}
    \item there do not exist solutions to problem \eqref{eq:mainPBBN} if
 $\lambda \geq \lambda_1$;
  \item for every $0<\lambda\leq \lambda_{1,s}$
  there do no exist solutions to problem \eqref{eq:mainPBBN} belonging
  to the closed ball $\mathcal{B}\subseteq L^{2^*}(\R^n)$ defined as
  $$\mathcal{B} := \{u\in L^{2^*}(\R^n):\,\|u\|\leq \mathcal{S}_n^{\,(n-2)/4}\}.$$
  \end{enumerate}
 \end{theorem}
We believe that a similar result should hold for sublinear perturbations (i.e. $0<p<1$), 
 but we will come back to this in a future work.
 \medskip
 
 As regards the case of the \emph{superlinear perturbation}, instead,
 the situation
 is quite different: as a matter of fact, 
 we can adapt
 the variational argument in \cite{BCSS} (based
 on the Mountain Pass Theorem,) to prove the following result.

\begin{theorem}\label{MS:3232}
Let $n \geq 3$ and $p\in(1,2^*-1)$. Set\footnote{Notice that, by definition, we have $\kappa_{s,n},\,\beta_{p,n} > 0$.}
\begin{equation} \label{eq:defkappabeta}
 \kappa_{s,n} := \min\{2-2s,n-2\},\qquad \beta_{p,n} := n-\frac{(p+1)(n-2)}{2}
 \end{equation}
 
 Then, the following assertions hold.
 \begin{itemize}
  \item[(1)] If $\kappa_{s,n} > \beta_{p,n}$, 
  then problem \eqref{eq:mainPBBN} admits a solution for every~$\lambda>0$.
  \item[(2)] If $\kappa_{s,n}\leq \beta_{p,n}$, 
  then problem \eqref{eq:mainPBBN} admits a solution
for $\lambda$ large enough.
 \end{itemize}
\end{theorem}
A `dichotomy' similar to that in Theorem \ref{MS:3232} appears
also in the purely local setting, see \cite[Section 2]{BrNir}, where
we have \emph{the same value of $\beta_{p,n}$} and
\begin{equation}\label{CACLKA}\kappa_{s,n} = \kappa_n = n-2.\end{equation}
The main difference between our mixed setting and the purely local one is that,
in our context, when $p\sim 1$ we can prove the existence of solutions to
problem \eqref{eq:mainPBBN} \emph{only for large values of $\lambda$}, independently
of how large the dimension is.

To clarify this phenomenon, let us assume that $n\geq 4$, so that $\kappa_{s,n} = 2-2s$.
Un\-der this assumption, the condition $\kappa_{s,n} > \beta_{s,p}$ boils down to
$$n > 2+\frac{4s}{p-1} =: \theta_{s,p}.$$
Since $\theta_{s,p}\to\infty$ as $p\to 1^+$, when $p\sim 1$ we 
have
$\kappa_{s,n} \leq \beta_{s,p}$,
and thus, by Theorem~\ref{MS:3232}, we deduce that problem \eqref{eq:mainPBBN} possesses
solutions for $\lambda$ sufficiently large.

In the purely local setting, instead, the situation is quite different:
indeed, by~\eqref{CACLKA}, since $\kappa_n = n-2$, we have
$$\kappa_n > \beta_{p,s}\,\,\Longleftrightarrow\,\,n > 2+\frac{4}{p+1} =: \theta_p.$$
Thus, since $\theta_p < 4$ for every $p > 1$, when~$n\geq 4$ we obtain in the classical case the
existence of solutions \emph{for every $\lambda > 0$}, independently of the exponent $p\in (1,2^*-1)$.

In any case, the restriction to large values of $\lambda$ for finding solutions of this type of problems
is a common occurrence also in the local scenario, see in particular
the case~$n=3$ in~\cite[Corollary~2.4]{BrNir}
(see also the nonlocal counterpart in~\cite{BCSS}).

We also remark that the case~$p=1$ in Theorem~\ref{thm:mainLinear} 
is structurally different than the case~$p\in(1,2^*-1)$ in Theorem~\ref{MS:3232}.
Indeed, on the one hand, in both cases we cannot establish the existence of solutions for $\lambda$ close to zero;
on the other hand, while Theorem~\ref{MS:3232} guarantees the existence of solutions for all $\lambda$ large enough,
Theorem~\ref{thm:mainLinear} only detects solutions for~$\lambda$ in a certain interval, showing also that no solutions
exist when~$\lambda$ is too large (therefore, the case~$p=1$ cannot be seen as a limit case
of the setting~$p\in(1,2^*-1)$).

\subsection{Plan of the paper}

The rest of this paper is organized as follows. Section~\ref{sec.Prel} contains the preliminary material needed to set up the appropriate functional spaces and to formalize the problems that we treat.

Then, in Section~\ref{SOINSEC}, we focus on the
mixed order Sobolev-type inequality and prove Theorems~\ref{thm:SnsSn} and~\ref{thm:Snneverachieved}.

Finally, the analysis of the critical problem~\eqref{eq:mainPBBN} occupies Section~\ref{sec:critPb},
where we prove Theorems~\ref{thm:nonexistence}, \ref{thm:mainLinear} and~\ref{MS:3232}.
\bigskip

\textbf{Acknowledgments.} We express our gratitude to the anonymous Referee for
the careful reading of the paper and for his/her
valuable comments, leading to this improved version of the manuscript.

 \section{The functional setting}\label{sec.Prel} 
  In this section we collect the notation and some
  preliminary results which will be used in the rest of the paper.
  More precisely, we introduce the adequate function spaces
  to study problem \eqref{eq:mainPBBN}, and we investigate
  the existence of extremals for some \emph{mixed Sobolev-type inequalities}.
  \medskip
  
Let~$s\in (0,1)$.
If $u:\R^n\to\R$ is a measurable function, we set
  \begin{equation} \label{eq:GagliardoSeminorm}
  [u]_s := \bigg(\iint_{\R^{2n}}\frac{|u(x)-u(y)|^2}{|x-y|^{n+2s}}\,dx\,dy\bigg)^{1/2},
  \end{equation}
  and we refer to $[u]_s$ as the \emph{Gagliardo seminorm} of $u$ (of order $s$).

  Let $\varnothing\neq \Omega\subseteq\R^n$ (with $n\geq 3$) be an \emph{arbitrary} open set,
  not necessarily bounded. We define the function space
  $\mathcal{X}^{1,2}(\Omega)$ as the com\-ple\-tion
  of $C_0^\infty(\Omega)$ with respect to the
  \emph{global norm} 
  $$\rho(u) := \big(\|\nabla u\|^2_{L^2(\R^n)}+[u]^2_s\big)^{1/2},\qquad u\in C_0^\infty(\Omega).$$
  \begin{remark} \label{rem:X12prop}
   A couple of observations concerning the space $\mathcal{X}^{1,2}(\Omega)$ are in order.
   \begin{enumerate}
    \item The norm $\rho(\cdot)$ is induced by the scalar product
    $$\langle u,v\rangle_{\rho} := \int_{\R^n}\nabla u\cdot\nabla v\,dx
    + \iint_{\R^{2n}}\frac{(u(x)-u(y))(v(x)-v(y))}{|x-y|^{n+2s}}\,dx\,dy,$$
where $\cdot$ denotes the usual scalar product in the Euclidean space
    $\R^n$, and $\mathcal{X}^{1,2}(\Omega)$
    is a \emph{Hilbert space}.
    \medskip
    
    \item Even if the function $u\in C_0^\infty(\Omega)$ 
  \emph{identically vanishes outside~$\Omega$}, it is often still convenient to consider 
  in the definition of $\rho(\cdot)$ the $L^2$-norm of $\nabla u$
  \emph{on the whole of $\R^n$}, rather than restricted to~$\Omega$ (though of course the result would be the same): this is to stress that the elements in
  $\mathcal{X}^{1,2}(\Omega)$
  are functions defined \emph{on the entire space $\R^n$} and not only on~$\Omega$ (and this is consistent with the nonlocal
  nature of the operator $\LL$). The benefit of having this global functional setting
  is that these functions can be \emph{globally approximated on $\R^n$} 
  (with respect to the norm $\rho(\cdot)$)
  by smooth functions with compact support in $\Omega$.

  \noindent
  In particular, when $\Omega\neq \R^n$, we will
  see that this \emph{global} definition of 
  $\rho(\cdot)$ implies that the functions in $\mathcal{X}^{1,2}(\Omega)$ naturally
  satisfy the nonlocal Dirichlet condition 
  prescribed in problem \eqref{eq:mainPBBN}, that is,
  \begin{equation} \label{eq:nonlocalDirX12}
   \text{$u\equiv 0$ a.e.\,in $\R^n\setminus\Omega$ for every $u\in\mathcal{X}^{1,2}(\Omega)$}.
   \end{equation}
   \end{enumerate}
  \end{remark}
  In order to better understand the \emph{nature} of the space
  $\mathcal{X}^{1,2}(\Omega)$, we distinguish two cases.
  \medskip
  
  (i)\,\,If \emph{$\Omega$ is bounded}. In this case we first recall the following
  inequality, which
   expresses the \emph{continuous embedding}
   of $H^1(\R^n)$ into $H^s(\R^n)$ (see, e.g., \cite[Proposition~2.2]{DRV}):
   there exists a constant $\mathbf{c} = \mathbf{c}(s) > 0$ such that,
   for every $u\in C_0^\infty(\Omega)$, one has
   \begin{equation} \label{eq:embeddingH1Hs}
    [u]^2_s \leq \mathbf{c}(s)\|u\|_{H^1(\R^n)}^2 = \mathbf{c}(s)\big(\|u\|^2_{L^2(\R^n)}
   + \|\nabla u\|^2_{L^2(\R^n)}\big).
   \end{equation}
   This, together with the classical \emph{Poincar\'e inequality}, implies that
   $\rho(\cdot)$ and the full $H^1$-norm in $\R^n$
   are \emph{actually equivalent} on the space $C^\infty_0(\Omega)$, and hence
   \begin{align*}
    \mathcal{X}^{1,2}(\Omega) & = \overline{C_0^\infty(\Omega)}^{\,\,\|\cdot\|_{H^1(\R^n)}} \\
    & = \{u\in H^1(\R^n):\,\text{$u|_\Omega\in H_0^1(\Omega)$ and 
    $u\equiv 0$ a.e.\,in $\R^n\setminus\Omega$}\}.
   \end{align*}
   
   (ii)\,\,If \emph{$\Omega$ is unbounded}. In this case, even if the \emph{embedding inequality}
   \eqref{eq:embeddingH1Hs} is still satisfied, the Poincar\'e inequality
   \emph{does not hold}; hence, the norm $\rho(\cdot)$ is no more equivalent to the full $H^1$-norm in $\R^n$,
   and $\mathcal{X}^{1,2}(\Omega)$ \emph{is not a subspace of $H^1(\R^n)$}.
   
   On the other had, by the classical \emph{Sobolev inequality} we infer
   the existence of a constant $\mathcal{S} = \mathcal{S}_n > 0$, 
   \emph{independent of the open set $\Omega$}, such that
   \begin{equation} \label{eq:Sobolevmista}
    \mathcal{S}_n\|u\|^2_{L^{2^*}(\R^n)} \leq \|\nabla u\|^2_{L^2(\R^n)}
    \leq \rho(u)^2\qquad\text{for every $u\in C_0^\infty(\Omega)$}.
   \end{equation}
   From this, we deduce that every Cauchy sequence in $C_0^\infty(\Omega)$ (with respect
   to the norm $\rho(\cdot)$) is also a Cauchy sequence \emph{in the space $L^{2^*}(\R^n)$};
   as a consequence, since the functions in $C_0^\infty(\Omega)$ identically vanish out of $\Omega$,
   we obtain
   $$\mathcal{X}^{1,2}(\Omega) = 
   \{u\in L^{2^*}(\R^n):\,\text{$u\equiv 0$ a.e.\,in $\R^n\setminus\Omega$,
   $\nabla u\in L^2(\R^n)$ and $[u]_s < \infty$}\}.$$
   In particular, when $\Omega = \R^n$ we have
   $$\mathcal{X}^{1,2}(\R^n) = \{u\in L^{2^*}(\R^n):\,\text{$\nabla u\in L^2(\R^n)$
   and $[u]_s < \infty$}\}.$$
   \begin{remark} \label{rem:XembeddedLtwostar}
   We stress that the \emph{mixed
   Sobolev-type inequality} \eqref{eq:Sobolevmista} holds \emph{for every open set $\Omega\subseteq\R^n$}
   (bounded or not); thus, we always have
   \begin{equation} \label{eq:contEmbXL2star}
    \mathcal{X}^{1,2}(\Omega)\hookrightarrow L^{2^*}(\Omega).
   \end{equation}
   Furthermore, by exploiting
   the density of $C_0^\infty(\Omega)$ in $\mathcal{X}^{1,2}(\Omega)$,
   we can e\-xtend
   i\-ne\-qua\-lity
   \eqref{eq:Sobolevmista} to \emph{every function $u\in \mathcal{X}^{1,2}(\Omega)$}, thereby obtaining
   $$\mathcal{S}_n\|u\|_{L^{2^*}(\R^n)}^2 
    \leq \rho(u)^2 = \|\nabla u\|_{L^2(\R^n)}^2
    + [u]^2_s\quad\text{for every $u\in\mathcal{X}^{1,2}(\Omega)$}.$$
   \end{remark}
   
   \section{The mixed Sobolev-type inequality}\label{SOINSEC}
Now we further develop our analysis on the \emph{mixed Sobolev inequality}
   \eqref{eq:Sobolevmista},
   that is,
   $$\mathcal{S}_n\|u\|_{L^{2^*}(\R^n)}^2\leq \rho(u)^2\qquad\text{for every $u\in C_0^\infty(\Omega)$},$$
   with the aim of proving Theorems~\ref{thm:SnsSn} and~\ref{thm:Snneverachieved}.
   
   To this end, our first goal is to find the \emph{sharp constant} in 
 \eqref{eq:Sobolevmista}, namely,
 \begin{equation} \label{eq:SnsOmegaProvv}
  \mathcal{S}_{n,s}(\Omega) := 
  \inf\big\{\rho(u)^2:\,u\in C_0^\infty(\Omega)\cap \mathcal{M}\big\}, 
 \end{equation}
 where $\mathcal{M}$ is the \emph{unit sphere} in $L^{2^*}(\R^n)$, that is,
 $$\mathcal{M}:= \{u\in L^{2^*}(\R^n):\,\|u\|_{L^{2^*}(\R^n)} = 1\}.$$
 We stress that, since $C_0^\infty(\Omega)$ is \emph{dense}
 in the space $\mathcal{X}^{1,2}(\Omega)$ (with respect to the norm $\rho(\cdot)$), 
 by the continuous embedding \eqref{eq:contEmbXL2star} we have
 $$\mathcal{S}_{n,s}(\Omega) = \inf\big\{\rho(u)^2:\,\text{$u\in\mathcal{X}^{1,2}(\Omega)
 \cap \mathcal{M}$}\big\}.$$
 Moreover, $\mathcal{S}_{n,s}(\Omega)$ is \emph{translation-invariant} with respect to~$\Omega$, that is,
 $$\mathcal{S}_{n,s}(x_0+\Omega) = \mathcal{S}_{n,s}(\Omega)\qquad\text{for every $x_0\in\R^n$}.$$
 \begin{remark} \label{rem:propSnrecall}
 Before starting with the study of the minimization problem \eqref{eq:SnsOmegaProvv},
 we review here below some well-known properties
 of the \emph{best constant} $\mathcal{S}_n$ in the classical Sobolev
 inequality, which is defined as follows: 
 $$\mathcal{S}_n := 
 \inf\big\{\|\nabla u\|^2_{L^2(\Omega)}:\,u\in C^\infty_0(\Omega)\cap \mathcal{M}\big\}.$$
 For a  proof of these facts we refer to 
 \cite{Aubin, Talenti}
 (see also \cite[Section~1]{BrNir}).
 \begin{enumerate}
  \item The constant $\mathcal{S}_n$ is \emph{independent of the open set $\Omega$}, 
  and depends only on the dimension $n$, 
  as stated in~\eqref{eq:bestSnloc}.
  \vspace*{0.1cm}
  
  \item Given any open set $\Omega\subseteq\R^n$ (bounded or not), we have
  $$\mathcal{S}_n = \inf\big\{\|\nabla u\|^2_{L^2(\R^n)}:\,\text{$u\in\mathcal{D}_0^{1,2}(\Omega)
  \cap \mathcal{M}$}\big\},$$
  where $\mathcal{D}_0^{1,2}(\Omega)$ is the space
  defined as the completion of $C_0^\infty(\Omega)$ with respect to the gradient norm
  $u\mapsto \mathcal{G}(u) := \|\nabla u\|_{L^2(\R^n)}$. 
  \vspace*{0.1cm}
  
  \item If $\Omega\subseteq\R^n$ is a \emph{bounded open set}, then $\mathcal{S}_n$ is 
  \emph{never achieved}.
  \vspace*{0.1cm}
  
  \item If $\Omega = \R^n$, then $\mathcal{S}_n$ 
  is not achieved in $C_0^\infty(\R^n)$
  but
  \emph{it is achieved in the big\-ger space $\mathcal{D}^{1,2}_0(\R^n)$}
  by the family of functions
  $$\mathcal{F} = \big\{U_{t,x_0}(x) = t^{\frac{2-n}{2}}\mathcal{U}((x-x_0)/t):\,\text{$t > 0,\,x_0\in\R^n$}
  \big\},$$
  where \begin{equation*}
  \mathcal{U}(z) := c\,(1+|z|^2)^{\frac{2-n}{2}}\end{equation*} and $c > 0$ is such that 
  $\|\mathcal{U}\|_{L^{2^*}(\R^n)} = 1$.
 \end{enumerate}
 \end{remark}
 We are thus ready to prove Theorem~\ref{thm:SnsSn}.
 
 \begin{proof}[Proof of Theorem~\ref{thm:SnsSn}]
Since $\rho(u)\geq \|\nabla u\|_{L^2(\Omega)}$ for every $u\in C^\infty_0(\Omega)$,
  we have 
  $$\mathcal{S}_{n,s}(\Omega)\geq \inf\big\{\|\nabla u\|^2_{L^2(\Omega)}:\,u\in C_0^\infty(\Omega)
  \cap\mathcal{M}\big\}
  = \mathcal{S}_n.$$
  To prove the reverse inequality, by the \emph{translation-invariance} of $\mathcal{S}_{n,s}(\Omega)$
  we assume that $x_0 = 0\in \Omega$, and we let $r > 0$ be such that $B_r(0)\subseteq\Omega$.
  We now observe that, given any $u\in C_0^\infty(\R^n)\cap \mathcal{M}$,
  there exists $k_0 = k_0(u)\in\N$ such that 
  $$\mathrm{supp}(u)\subseteq B_{kr}(0)\quad\text{for every $k\geq k_0$};$$
  as a consequence, setting $u_k := k^{\frac{n-2}{2}}u(kx)$ (for $k\geq k_0$), we readily see that
  $$\mathrm{supp}(u_k)\subseteq B_r(0)\subseteq \Omega\quad\text{and}\quad
  \|u_k\|_{L^{2^*}(\R^n)} = 1.$$
  Taking into account the definition of $\mathcal{S}_{n,s}(\Omega)$, we find that, for every $k\geq k_0$,
  \begin{align*}
   \mathcal{S}_{n,s}(\Omega) & \leq \rho(u_k)^2 = 
   \|\nabla u_k\|^2_{L^2(\R^n)}+[u_k]^2_s
= \|\nabla u\|^2_{L^2(\R^n)}+k^{2s-2}[u]^2_s .
  \end{align*}
  From this, letting $k\to\infty$ (and recalling that $s < 1$), we obtain
  $$\mathcal{S}_{n,s}(\Omega)\leq \|\nabla u\|^2_{L^2(\R^n)}.$$
  By the arbitrariness of $u\in C^\infty_0(\Omega)\cap \mathcal{M}$ and the fact that
   $\mathcal{S}_n$ \emph{is independent of the open set $\Omega$}
   (see Remark \ref{rem:propSnrecall}-(1)), we finally infer that
   $$\mathcal{S}_{n,s}(\Omega)\leq \inf\big\{\|\nabla u\|^2_{L^2(\R^n)}:\,C^\infty_0(\Omega)\cap \mathcal{M}
   \big\} = \mathcal{S}_n,$$
   and hence $\mathcal{S}_{n,s}(\Omega) = \mathcal{S}_n$.
 \end{proof}
 
 We are now in the position of proving Theorem~\ref{thm:Snneverachieved}.
 
 \begin{proof}[Proof of Theorem~\ref{thm:Snneverachieved}]
  Arguing by contradiction, let us suppose that there exists a nonzero
  function $u_0\in\mathcal{X}^{1,2}(\Omega)$
  such that 
  $\|u_0\|_{L^{2^*}(\R^n)} = 1$ and
  $$\rho(u_0)^2 = \|\nabla u_0\|^2_{L^2(\R^n)}+[u_0]^2_s = \mathcal{S}_{n}.$$
  Taking into account that $\mathcal{X}^{1,2}(\Omega)\subseteq\mathcal{D}_0^{1,2}(\Omega)$
  (this inclusion being a straightforward of the fact that $\rho(\cdot)\geq \mathcal{G}(\cdot)$,
  see Remark \ref{rem:propSnrecall}-(2)),
  we infer that
  $$\mathcal{S}_n \leq \|\nabla u_0\|^2_{L^2(\Omega)} \leq \|\nabla u_0\|^2_{L^2(\R^n)}+[u_0]^2_s
  = \rho(u_0)^2 = \mathcal{S}_n,$$
  from which we derive that $[u_0]_s = 0$.
  As a consequence, the function 
  $u_0$ must be \emph{constant in $\R^n$}, but this is contradiction with the fact that
  $\|u_0\|_{L^{2^*}(\R^n)} = 1$.
  \end{proof}
  \begin{remark} \label{rem:Snasymptotic}
  Even if Theorem \ref{thm:Snneverachieved} shows that
   the constant $\mathcal{S}_{n,s}(\Omega) = \mathcal{S}_n$ is
  \emph{never achieved} in the space $\mathcal{X}^{1,2}(\Omega)$ 
  (independently of the open set $\Omega$), in the particular case $\Omega = \R^n$
  we can prove that $\mathcal{S}_n$ is achieved `in the limit':
  more precisely, if $\mathcal{F} = \{U_{t,x_0}\}$ is as in Remark \ref{rem:propSnrecall}-(4), we have
  $$\rho(U_{t,x_0})^2 \to \mathcal{S}_n\quad\text{as $t\to\infty$}.$$
Indeed, 
$$ 
|\mathcal{U}(z) |\le C\,\min\{1,|z|^{2-n}\}\qquad{\mbox{and}}\qquad
|\nabla\mathcal{U}(z) |\le C\,\min\{|z|, \,|z|^{1-n}\},$$
for some~$C>0$, therefore (up to renaming~$C$ line after line)
\begin{eqnarray*}
[\mathcal{U}(z)]_s^2&=&
\iint_{\R^{2n}}\frac{|\mathcal{U}(x+y)-\mathcal{U}(x)|^2}{|y|^{n+2s}}\,dx\,dy\\&\le&
\iint_{\R^{n}\times B_1}
\left|
\int_0^1\nabla \mathcal{U}(x+ty)\cdot y\,dt\right|^2 \frac{dx\,dy}{|y|^{n+2s}}\\&&\qquad
+2\iint_{\R^{n}\times(\R^n\setminus B_1)}\big(|\mathcal{U}(x+y)|^2+|\mathcal{U}(x)|^2\big)\frac{dx\,dy}{|y|^{n+2s}}\\&\le&
\iiint_{\R^{n}\times B_1\times(0,1)}\min\{|x+ty|^2, \,|x+ty|^{2(1-n)}\}
\,\frac{dx\,dy\,dt}{|y|^{n+2s-2}}\\&&\qquad+4\iint_{\R^{n}\times(\R^n\setminus B_1)}
|\mathcal{U}(z)|^2\frac{dz\,dy}{|y|^{n+2s}}
\\&\le&\iiint_{\R^{n}\times B_1\times(0,1)}\min\{|z|^2, \,|z|^{2(1-n)}\}
\,\frac{dz\,dy\,dt}{|y|^{n+2s-2}}\\&&\qquad
+C\iint_{\R^{n}\times(\R^n\setminus B_1)}
\min\{1,|z|^{2(2-n)}\}\frac{dz\,dy}{|y|^{n+2s}},
\end{eqnarray*}
which is finite.

Hence, $\mathcal{U}\in\mathcal{X}^{1,2}(\R^n)$ and consequently~$U_{t,x_0}\in\mathcal{X}^{1,2}(\R^n)$ for every $t > 0$ and $x_0\in\R^n$.

Moreover, recalling that
  $$U_{t,x_0}(x) = t^{\frac{2-n}{2}}\mathcal{U}\Big(\frac{x-x_0}{t}\Big)\quad\text{and}\quad
  \|U_{t,x_0}\|_{L^{2^*}(\R^n)} = \|\mathcal{U}\|_{L^{2^*}(\R^n)} = 1,$$
  by arguing as in the proof of Theorem \ref{thm:SnsSn} we have
  \begin{align*}
   \rho(U_{t,x_0})^2 = \rho(U_{t,0})^2 = \|\nabla \mathcal{U}\|^2_{L^2(\R^n)}+t^{2s-2}[\mathcal{U}]^2_s.
  \end{align*}
  From this, since $\mathcal{U} = U_{1,0}$ is an optimal function in the classical
  Sobolev inequality (and since $s < 1$), by letting $t\to\infty$ we obtain
  $$\rho(U_{t,x_0})^2\to \|\nabla \mathcal{U}\|_{L^2(\R^n)}^2 = \mathcal{S}_n.$$
   \end{remark}
   
 \section{Critical problems driven by \texorpdfstring{$\mathcal{L}$}{L}} \label{sec:critPb}
 We now develop our study concerning the solvability of 
 problem \eqref{eq:mainPBBN}.
 Throughout this section, we tacitly understand that
 \begin{itemize}
  \item[(i)] $\Omega\subseteq\R^n$ is a \emph{bounded open set};
  \item[(ii)] $\lambda\in\R$;
  \item[(iii)] $1\leq p < 2^*-1$.
 \end{itemize}
 Moreover, we inherit all the
 definitions and notation of 
 Sections \ref{sec.Prel} and~\ref{SOINSEC}.
We also introduce the convex cone
 $$\mathcal{X}^{1,2}_+(\Omega) := \{u\in\mathcal{X}^{1,2}(\Omega):\,\text{$u\geq 0$ a.e.\,in $\Omega$}\}.$$
 We begin by giving the precise definition of \emph{solution} to \eqref{eq:mainPBBN}.
 \begin{definition} \label{def:solPb}
  We say that a function $u\in\mathcal{X}^{1,2}_+(\Omega)$ is a 
  \emph{solution} to problem \eqref{eq:mainPBBN} if it
  sa\-ti\-sfies the following properties:
  \begin{enumerate}
   \item $|\{x\in\Omega:u(x) > 0\}| > 0$;
   \item for every test function $v\in\mathcal{X}^{1,2}(\Omega)$ we have
   \begin{equation} \label{eq:weakSol}
   \begin{split}
    & \int_{\Omega}\nabla u\cdot\nabla v\,dx 
    +\iint_{\R^{2n}}\frac{(u(x)-u(y))(v(x)-v(y))}{|x-y|^{n+2s}}\,dx\,dy \\
    & \qquad = \int_{\Omega}(u^{2^*-1}+\lambda u^p)v\,d x.
    \end{split}
   \end{equation}
  \end{enumerate}
 \end{definition}

We observe that, since $\Omega\subseteq\R^n$ is bounded, 
 $$\mathcal{X}^{1,2}(\Omega) = \{u\in H^1(\R^n):\,\text{$u|_\Omega\in H_0^1(\Omega)$ and 
    $u\equiv 0$ a.e.\,in $\R^n\setminus\Omega$}\}.$$
 As a consequence, the assumption $u\in\mathcal{X}_+^{1,2}(\Omega)$ \emph{contains}
 both the Dirichlet boundary con\-di\-tion $u\equiv 0$ a.e.\,in $\R^n\setminus\Omega$ and
 the sign condition $u\geq 0$ a.e.\,in $\Omega$.

Moreover, all the integrals in \eqref{eq:weakSol} are finite, due to the embedding $\mathcal{X}^{1,2}(\Omega) \hookrightarrow 
 L^{2^*}(\R^n)$ and the fact that~$p< 2^*-1$.

 Before proceeding we highlight that, if $u\in\mathcal{X}^{1,2}_+(\Omega)$
 is a given function, the validity of identity \eqref{eq:weakSol}
 (which expresses the fact that $u$ is a solution to pro\-blem
 \eqref{eq:mainPBBN}) is \emph{actually equivalent} to the following condition
 $$\d\mathcal{F}_{\lambda, p}(u)\equiv 0\quad\text{on $\mathcal{X}^{1,2}(\Omega)$},$$
 where $\mathcal{F}_{\lambda, p}$ is the $C^1$-functional defined on the Hilbert space
 $\mathcal{X}^{1,2}(\Omega)$ by
 \begin{equation} \label{eq:defFlambda}
 \mathcal{F}_{\lambda, p}(u) := \frac{1}{2}\rho(u)^2
 -\frac{1}{2^*}\int_{\Omega}|u|^{2^*}\,dx-\frac{\lambda}{p+1}\int_\Omega|u|^{p+1}\,d x .
 \end{equation}
 Thus, we derive that the solutions to problem \eqref{eq:mainPBBN} are precisely
 the \emph{nonnegative critical points} of the functional $\mathcal{F}_{\lambda,p}$.
 This characterization will be used, as a crucial tool, in order
 to study the existence of solutions to problem \eqref{eq:mainPBBN}.
 \medskip
 
 With Definition \ref{def:solPb} at hand, 
 we establish the following boundedness and positivity result.
 
 \begin{theorem} \label{thm:aprioriPropSol}
  Assume that there exists
  a solution $u_0\in\mathcal{X}^{1,2}_+(\Omega)$ to pro\-blem 
  \eqref{eq:mainPBBN}
for some $\lambda \in\R$ and $p\in[1,2^*-1)$.
  Then, the following facts hold:
  \begin{enumerate}
   \item $u_0\in L^\infty(\R^n)$;
   \item if, in addition, $\lambda \geq 0$, then $u_0 > 0$ a.e.\,in $\Omega$.
  \end{enumerate}
 \end{theorem}
 \begin{proof}
  The boundedness of $u_0$ is a direct consequence of \cite[Theorem~4.1]{BMV}, 
  applied here with the choice $f(x,t) := t^{2^*-1}+\lambda t^p$. Indeed, $u_0$ is a solution to
  $$\begin{cases}
  \LL u = f(x,u) & \text{in $\Omega$}, \\
  u \gneqq 0 & \text{in $\Omega$}, \\
  u \equiv 0 & \text{in $\R^n\setminus\Omega$},
  \end{cases}$$
  and the function $f$ satisfies the following properties:
  \begin{itemize}
   \item[(i)] $f:\Omega\times[0,\infty)\to\R$ is a Carath\'{e}odory function;
   \vspace*{0.05cm}
   
   \item[(ii)] $f(\cdot,t)\in L^\infty(\Omega)$ for every $t\geq 0$;
   \vspace*{0.05cm}
   
   \item[(iii)] $|f(x,t)| \leq (1+|\lambda|)(1+t^{2^*-1})$ for every $t\geq 0$.
  \end{itemize}
  In view of these properties of $f$, we are then entitled to invoke \cite[Theorem~4.1]{BMV}
  (see also  \cite[Remark~4.2]{BMV}), ensuring that $u_0\in L^\infty(\Omega)$. This,
  jointly with the fact that 
  $u_0\equiv 0$ a.e.\,in $\R^n\setminus\Omega$,
  implies
  that $u_0\in L^\infty(\R^n)$, as desired.
  \medskip
  
  Assume now that $\lambda \geq 0$. Recalling that $u_0$ is a solution
  to \eqref{eq:mainPBBN} (in the sense of De\-fi\-ni\-tion \ref{def:solPb}),
  and since $u_0\geq 0$ a.e.\,in $\Omega$, for every $v\in\mathcal{X}^{1,2}_+(\Omega)$ we have
  \begin{equation} \label{eq:LLugeqzeroweak}
  \begin{split}
     & \int_{\Omega}\nabla u_0\cdot\nabla v\,dx 
    +\iint_{\R^{2n}}\frac{(u_0(x)-u_0(y))(v(x)-v(y))}{|x-y|^{n+2s}}\,dx\,dy \\
    & \qquad = \int_{\Omega}(u_0^{2^*-1}+\lambda u_0^p)v\,d x \geq 0.
    \end{split}
  \end{equation}
  As a consequence, we can apply the Strong Maximum Principle
  in \cite[Theorem~3.1]{BMV} (this time with the choice $f(x,t)\equiv 0$), ensuring that
  $$\text{either $u_0 > 0$ a.e.\,in $\Omega$\quad or\quad $u_0\equiv 0$ a.e.\,in $\Omega$}.$$
  On the other hand, since $|\{u_0 > 0\}| > 0$, 
  we have that $u_0$ is not identically va\-ni\-shing in $\Omega$; thus,
  $u_0 > 0$ a.e.\,in $\Omega$,
  and the proof is complete.
 \end{proof}
 Thanks to Theorem \ref{thm:aprioriPropSol}, we can now prove
 Theorem~\ref{thm:nonexistence}.
 \begin{proof}[Proof of Theorem~\ref{thm:nonexistence}]
  Arguing by contradiction, let us suppose that there exists a nonzero
  solution $u_0$ to problem \eqref{eq:mainPBBN}. Owing to Theorem \ref{thm:aprioriPropSol},
  we have 
  $u_0\in L^\infty(\R^n)$;
  moreover, defining $f(x,t) := |t|^{2^*-2}t+\lambda |t|^{p-1}t$ (with $t\in\R$), 
  we see that $u_0$ solves
  \begin{equation} \label{eq:pbformaRosSerra}
  \begin{cases}
  \LL u = f(x,u) & \text{in $\Omega$}, \\
  u \equiv 0 & \text{in $\R^n\setminus\Omega$},
  \end{cases}
  \end{equation}
  and the nonlinearity $f$ satisfies the following properties:
  \begin{itemize}
   \item[(i)] $f\in C^{0,1}_{\mathrm{loc}}(\overline{\Omega}\times\R)$;
   \vspace*{0.05cm}
   
   \item[(ii)] for every $x\in\Omega$ and every $t\in\R$, we have
   \begin{align*}
    \frac{n-2}{2}\,tf(x,t) & = \frac{n-2}{2}\big(|t|^{2^*}+\lambda |t|^{p+1}\big)
    = n\bigg(\frac{|t|^{2^*}}{2^*}+\frac{\lambda |t|^{p+1}}{2^*}\bigg) \\
    & (\text{since $\lambda\leq 0$ and $p\leq 2^*-1$}) \\
    & \geq n\bigg(\frac{|t|^{2^*}}{2^*}+\frac{\lambda |t|^{p+1}}{p+1}\bigg)
    = n\int_0^t f(x,\tau)\,d\tau.
   \end{align*}
  \end{itemize}
  As a consequence, since $u_0$ is a \emph{bounded solution} of 
  \eqref{eq:pbformaRosSerra}, $f$ satisfies (i)-(ii)
  and $\Omega$ is star-shaped,
  we are entitled to apply \cite[Theorem~1.3]{RSPoho}, ensuring that 
  $$\text{$u_0\equiv 0$ a.e.\,in $\Omega$}.$$
  This is clearly a contradiction with the fact that $u_0\gneqq 0$ in $\Omega$
  (recall that $u_0$ is a solution to \eqref{eq:mainPBBN}, see Definition \ref{def:solPb}),
  and the proof is complete.
 \end{proof}
{F}rom now on we assume that
 \begin{equation} \label{eq:lambdapositive}
  \lambda > 0.
 \end{equation}
 Under this assumption
 (which is somewhat in agreement with Theorem \ref{thm:nonexistence})
  and taking into account all the discussion
 carried out so far, we can start our study of the \emph{solvability} of
 problem \eqref{eq:mainPBBN}. To this end, since the \emph{linear case} $p = 1$
 and the \emph{superlinear case} $p > 1$ present some subtle differences,
 it is worth treating these cases separately.
 \subsection{The linear case \texorpdfstring{$p = 1$}{p = 1}} \label{sec:linear}
 We begin by studying the solvability of \eqref{eq:mainPBBN} in the
 \emph{linear case $p = 1$}, that is, we begin by considering the problem
 \begin{equation} \label{eq:mainPBlinear}
  \begin{cases}
  \LL u = u^{2^*-1}+\lambda u & \text{in $\Omega$}, \\
  u \gneqq 0 & \text{in $\Omega$}, \\
  u \equiv 0 & \text{in $\R^n\setminus\Omega$},
  \end{cases}
 \end{equation}
 In this \emph{linear} context, the existence of solutions
 to problem \eqref{eq:mainPBlinear} for a given $\lambda > 0$
 is deeply influenced
 by the \emph{first Dirichlet eigenvalues} of $(-\Delta)^s$ and of $\LL$.
 
 We recall that, given the bounded open set $\Omega$, one defines
 \begin{itemize}
  \item[(i)] the first Dirichlet eigenvalue of $(-\Delta)^s$ in $\Omega$ as
  \begin{equation} \label{eq:deflambda1s}
   \lambda_{1,s} := 
  \inf\big\{[u]^2_s:\,\text{$u\in C_0^\infty(\Omega)$ and $\|u\|_{L^2(\R^n)} = 1$}\big\};
  \end{equation}
  \item[(ii)] the first Dirichlet eigenvalue of $\LL$ in $\Omega$ as
  \begin{equation} \label{eq:deflambda1}
   \lambda_{1} := 
  \inf\big\{\rho(u)^2:\,\text{$u\in C_0^\infty(\Omega)$ and $\|u\|_{L^2(\R^n)} = 1$}\big\}.
  \end{equation}
 \end{itemize}
 Clearly, both $\lambda_{1,s}$ and $\lambda_1$ depend on the open set $\Omega$;
 however, in order to simplify the notation, we avoid to make explicit this dependence
 in what follows.
 
 For the sake
 of completeness (and for a future reference), 
 we then present in the next remark a short
 overview of the main properties
 of $\lambda_{1,s}$ and $\lambda_{1}$.
 \begin{remark}[Properties of $\lambda_{1,s}$ and $\lambda_{1}$] \label{rem:problambdalambdas}
 As regards $\lambda_{1,s}$ we first
  observe that, since the map 
  $u\mapsto [u]_s =: \mathcal{N}(u)$ is
 a norm on $C^\infty_0(\Omega)$
 which is equivalent to the full $H^s$-norm
 (see, e.g., \cite[Theorem~6.5]{DRV} and recall that $\Omega$ is bounded), one has
 $$\lambda_{1,s} = 
 \inf\big\{[u]^2_s:\,\text{$u\in \mathcal{D}_0^{s,2}(\Omega)$ and $\|u\|_{L^2(\R^n)} = 1$}\big\},$$
 where $\mathcal{D}_0^{s,2}(\Omega)\subseteq L^2(\R^n)$ is the completion of
 $C_0^{\infty}(\Omega)$ with respect to $\mathcal{N}$.
 Moreover, since the embedding $\mathcal{D}_0^{s,2}(\Omega)\hookrightarrow L^2(\R^n)$
 is compact, 
 it is not difficult to see that $\lambda_{1,s}$
 is \emph{actually achieved} in this larger space $\mathcal{D}_0^{s,2}(\Omega)$, that is,
 $$\exists\,\,\phi_0\in\mathcal{D}_0^{s,2}(\Omega):\,\text{$\|\phi_0\|_{L^2(\R^n)} = 1$
 and $[\phi_0]^2 = \lambda_{1,s} > 0$}.$$
 As a matter of fact, the above function $\phi_0$ can be chosen to be \emph{strictly positive
 a.e.\,in $\Omega$} (see, e.g., \cite[Proposition~9]{SerValEigen}); 
 in addition, since $\phi_0$ is a constrained minimizer of the functional $u\mapsto [u]_s^2$,
 by the Lagrange Multiplier Rule we have that $\phi_0$ is weak solution to the \emph{eigenvalue problem}
 \begin{equation*}
  \begin{cases}
 (-\Delta)^su = \lambda_{1,s} u & \text{in $\Omega$}, \\
 u \not\equiv 0 & \text{in $\Omega$}, \\
 u\equiv 0 & \text{in $\R^n\setminus\Omega$}.
 \end{cases}
 \end{equation*}
 As regards $\lambda_{1}$, we have completely analogous
 results (with the obvious modifications): first of all, since $\rho(\cdot)$
 defines
 a norm on $C^\infty_0(\Omega)$
 which is e\-qui\-va\-lent to the full $H^1$-norm in $\R^n$
 (as $\Omega$ is bounded), we have
 $$\lambda_{1} = 
 \inf\big\{\rho(u)^2:\,\text{$u\in \mathcal{X}^{1,2}(\Omega)$ and $\|u\|_{L^2(\R^n)} = 1$}\big\}.$$
 Moreover, since the embedding $\mathcal{X}^{1,2}(\Omega)\hookrightarrow L^2(\R^n)$
 is compact, it is not difficult to see that $\lambda_{1}(\Omega)$
 is \emph{actually achieved} in this larger space $\mathcal{X}^{1,2}(\Omega)$, that is,
 $$\exists\,\,\psi_0\in\mathcal{X}^{1,2}(\Omega):\,\text{$\|\psi_0\|_{L^2(\R^n)} = 1$
 and $\rho(\psi_0)^2 = \lambda_{1} > 0$}.$$
 As a matter of fact, the above function $\psi_0$ can be chosen to be \emph{strictly positive
 a.e.\,in $\Omega$}
 (see, e.g., \cite[Proposition~5.1]{BMV}); in addition, 
 since $\psi_0$ is a constrained minizer of the functional
 $\rho(\cdot)^2$, by the Lagrange Multiplier Rule we easily see that
 $\psi_0$ is weak solution to the \emph{eigenvalue problem}
 \begin{equation*}
  \begin{cases}
 \LL u = \lambda_{1} u & \text{in $\Omega$}, \\
 u \not\equiv 0 & \text{in $\Omega$}, \\
 u\equiv 0 & \text{in $\R^n\setminus\Omega$}.
 \end{cases}
 \end{equation*}
 \end{remark}

 We are now 
 approaching the proof of Theorem \ref{thm:mainLinear}. We will achieve this
 through several independent results.
 To begin with, we prove a lemma linking the existence of solutions
 to \eqref{eq:mainPBlinear} with the existence of constrained minimizers
 for a suitable functional.
 
 As a matter of facts, since solutions to 
 problem \eqref{eq:mainPBBN} 
 are precisely the \emph{unconstrained} nonnegative critical points
 of the functional $\mathcal{F}_{\lambda,p}$ defined in \eqref{eq:defFlambda} (for all $p\geq 1$),
 we now aim at proving that, in this \emph{linear context}, there is a link
 between \emph{existence of solutions}
 to \eqref{eq:mainPBlinear} and \emph{existence of minimizers}
 for the functional
 \begin{equation} \label{eq:QformQlambda}
  \mathcal{Q}_\lambda(u) := \rho(u)^2-\lambda\|u\|^{2}_{L^2(\R^n)}\qquad (u\in\mathcal{X}^{1,2}(\Omega)),
  \end{equation}
 constrained to the manifold $\mathcal{V}(\Omega) := \mathcal{X}^{1,2}(\Omega)\cap\mathcal{M}$.
 \begin{lemma} \label{lem:solMinLinear}
  For every fixed $\lambda > 0$, we define
  \begin{equation} \label{eq:defSnlambdaLin}
   \mathcal{S}(\lambda) := \inf_{u\in\mathcal{V}(\Omega)}\mathcal{Q}_\lambda(u).
  \end{equation}
  We assume that $\mathcal{S}(\lambda) > 0$
  and that $\mathcal{S}(\lambda)$ \emph{is achieved}, that is, there exists 
  some function $w\in\mathcal{V}(\Omega)$
   such that
   $\mathcal{Q}_\lambda(w) = \mathcal{S}(\lambda)$.
   
   Then, there exists a solution to \eqref{eq:mainPBlinear}.
 \end{lemma}
 Before giving the proof of Lemma \ref{lem:solMinLinear}, we list in the
 next remark some properties of the number $\mathcal{S}(\lambda)$ which will
 be used in the sequel.
 \begin{remark}[Properties of $\mathcal{S}(\lambda)$] \label{rem:propSnlambda}
  We begin by observing that,
  by H\"older's ine\-quality, for every $u\in\mathcal{V}(\Omega)$
  we have
 \begin{align*}
  \mathcal{Q}_\lambda(u) & \geq \inf_{u\in\mathcal{V}(\Omega)}\rho^2(u)-\lambda\|u\|^2_{L^2(\R^n)}
   \geq \mathcal{S}_n-\lambda\|u\|^2_{L^2(\R^n)} \\
   & (\text{recall that $u\equiv 0$ a.e.\,in $\R^n\setminus\Omega$}) \\
   & \geq \mathcal{S}_n
  -\lambda |\Omega|^{4/n}.
 \end{align*}
 As a consequence, we obtain
 $$\mathcal{S}(\lambda) \geq  \mathcal{S}_n
  -\lambda |\Omega|^{4/n} > -\infty\quad\text{for every $\lambda > 0$}.$$
 In addition, from the definition of $\mathcal{S}(\lambda)$ we easily infer that
 \begin{itemize}
  \item[(i)] $\mathcal{S}(\lambda)\leq\mathcal{S}_n$ for every $\lambda > 0$;
  \vspace*{0.1cm}
  
  \item[(ii)] $\mathcal{S}(\lambda)
  \leq \mathcal{S}(\mu)$ for every $0 < \mu < \lambda$.
 \end{itemize}
 Finally, owing to the definition of $\lambda_1$ in \eqref{eq:deflambda1} 
 (and recalling that
 $\lambda_1$ is a\-chi\-e\-ved \emph{in the space $\mathcal{X}^{1,2}(\Omega)$},
 see Remark \ref{rem:problambdalambdas}), it is easy to see that
 $$\mathcal{S}(\lambda) \geq 0\,\,\Longleftrightarrow\,\,0<\lambda\leq \lambda_1.$$
 \end{remark}
 We are now ready to provide the proof of Lemma \ref{lem:solMinLinear}.
 \begin{proof}[Proof of Lemma \ref{lem:solMinLinear}]
   Let $w\in\mathcal{V}(\Omega)$ be a constrained minimizer for 
   the functional $\mathcal{Q}_\lambda$
   (whose existence is guaranteed by the assumptions), that is,
  $$\mathcal{Q}_\lambda(w) = \mathcal{S}(\lambda) > 0.$$
  Since $\mathcal{Q}_\lambda(|w|)\leq \mathcal{Q}_\lambda(w)$, 
  by possibly replacing $w$ with $|w|$ we may and do assume that $w\gneqq 0$ a.e.\,in $\Omega$;
  moreover, 
  by the Lagrange Multiplier Rule we have
  $$\d\mathcal{Q}_\lambda(w) \equiv \mu\,\d\big(u\mapsto \|u\|^2_{L^{2^*}(\R^n)}\big)(w)
  \quad\text{on $\mathcal{X}^{1,2}(\Omega)$},$$
  for some $\mu\in\R$.  This means that
  \begin{equation}  \label{eq:Lagrangew}
  \begin{split}
    & \int_{\Omega}\nabla w\cdot\nabla \phi
   + \iint_{\R^{2n}}\frac{(w(x)-w(y))(\phi(x)-\phi(y))}{|x-y|^{n+2s}} \\
   & \qquad\quad = \int_{\Omega}(\mu w^{2^*-1}+\lambda w)\phi\,d x
   \qquad\text{for every $\phi\in\mathcal{X}^{1,2}(\Omega)$}.
   \end{split}
  \end{equation}
  Now, if we choose $\phi = w$ in this last identity, we get
  \begin{align*}
   \mu = \mu\|w\|^{2^*}_{L^{2^*}(\R^n)} = \rho(w)^2-\lambda\|w\|^2_{L^2(\R^n)} = \mathcal{Q}_\lambda(w)
   = \mathcal{S}(\lambda) > 0.
  \end{align*}  
  As a consequence, 
  setting $u := \mathcal{S}(\lambda)^{(n-2)/4}w$, we see that 
  \medskip
  
  (i)\,\,$u\in\mathcal{X}_+^{1,2}(\Omega)\setminus\{0\}$, 
  since $w\gneqq 0$ in $\Omega$ and $\mathcal{S}(\lambda) > 0$;
  \vspace*{0.1cm}
  
  (ii)\,\,for every $\phi\in\mathcal{X}^{1,2}(\Omega)$, from \eqref{eq:Lagrangew} we get
  $$
  \int_{\Omega}\nabla u\cdot\nabla \phi
   + \iint_{\R^{2n}}\frac{(u(x)-u(y))(\phi(x)-\phi(y))}{|x-y|^{n+2s}}
   = \int_{\Omega}(u^{2^*-1}+\lambda u)\phi\,d x.
   $$
  Gathering (i)-(ii), we then conclude that $u$ is a solution to \eqref{eq:mainPBlinear}, as desired.
 \end{proof}
 On account of Lemma \ref{lem:solMinLinear}, an important information
 to study the solvability of problem \eqref{eq:mainPBlinear}
 is the \emph{sign} of the \emph{real number} $\mathcal{S}(\lambda)$.
 In this perspective, we prove the following result:
 \begin{lemma} \label{lem:SlambdaequalS}
  For every $0<\lambda\leq \lambda_{1,s}$, we have
  $$\mathcal{S}(\lambda) = \mathcal{S}_n > 0.$$
 \end{lemma}
 \begin{proof}
 Let $\lambda\in(0,\lambda_{1,s})$ be given. By the definition of $\lambda_{1,s}$ in \eqref{eq:deflambda1s},
 for every $u\in C_0^\infty(\Omega)\cap \mathcal{M}$ we have
 \begin{align*}
  \mathcal{Q}_{\lambda}(u) & = \|\nabla u\|^2_{L^2(\R^n)}
  + \big([u]_s^2-\lambda\|u\|^2_{L^2(\R^n)}\big) \\
  & \geq \|\nabla u\|^2_{L^2(\R^n)}
  + (\lambda_{1,s}-\lambda)\|u\|^2_{L^2(\R^n)} \\
  & \geq \|\nabla u\|^2_{L^2(\R^n)}.
 \end{align*}
 As a consequence, 
 \begin{align*}
 \mathcal{S}(\lambda) & = \inf\big\{\mathcal{Q}_\lambda(u):\,u\in C_0^\infty(\Omega)
 \cap \mathcal{M}\big\} \\
 & \geq \inf\big\{\|\nabla u\|^2_{L^2(\R^n)}:\,u\in C_0^\infty(\Omega)
 \cap \mathcal{M}\big\} = \mathcal{S}_n,
 \end{align*}
 from which we deduce that $\mathcal{S}(\lambda) = \mathcal{S}_n$, as desired.
 \end{proof}
 By combining Remark \ref{rem:propSnlambda} with Lemma \ref{lem:SlambdaequalS}, we 
conclude that
 \begin{itemize}
  \item[(i)] $\mathcal{S}(\lambda) = \mathcal{S}_n$ for every $0<\lambda\leq \lambda_{1,s}$;
  \vspace*{0.1cm}
  
  \item[(ii)] $\mathcal{S}(\lambda)\geq 0$ for every $0<\lambda\leq \lambda_1$;
  \vspace*{0.1cm}
  
  \item[(iii)] $\mathcal{S}(\lambda) < 0$ for every $\lambda > \lambda_1$.
 \end{itemize}
 We now turn to prove that $\mathcal{S}(\cdot)$ is 
 \emph{continuous function of $\lambda \in (0,\infty)$}; this, together with the above (i)-to-(iii),
 ensures that
 $\mathcal{S}(\cdot)$ assumes \emph{all the values}
 between $0$ and the constant $\mathcal{S}_n$ when $\lambda$ ranges in the interval $(0,\lambda_1]$.
 \begin{lemma} \label{lem:contSnlambda}
 The function $\lambda\mapsto \mathcal{S}(\lambda)$ is continuous on $(0,\infty)$.
 \end{lemma}
 \begin{proof}
  By Remark \ref{rem:propSnlambda}, $\mathcal{S}(\cdot)$
  is \emph{nonincreasing} on the interval $(0,\infty)$.
  
  As a consequence,
  for every $\lambda > 0$ we have
  $$\exists\,\,\,\mathcal{S}(\lambda-) := \lim_{\mu\to \lambda^-}\mathcal{S}(\mu)\in\R\quad
  \text{and}\quad \exists\,\,\,\mathcal{S}(\lambda+) := \lim_{\mu\to \lambda^+}\mathcal{S}(\mu)\in\R.$$ 
  To prove the continuity of $\mathcal{S}(\cdot)$ on $(0,\infty)$, we 
  then pick any $\lambda_0 > 0$ and we show that
  $\mathcal{S}(\cdot)$ is 
  con\-ti\-nuo\-us \emph{both from the left and from the right} at $\lambda_0$, 
  that is,
  $$\mathcal{S}(\lambda_0-) = \mathcal{S}(\lambda_0+) = \mathcal{S}(\lambda_0).$$
  As regards the continuity from the left, we proceed as follows: first of all,
  given any $\e > 0$, 
  by definition of $\mathcal{S}(\lambda_0)$ there exists 
  $u = u_{\e,\lambda_0}\in \mathcal{V}(\Omega)$ such that
  $$
   \mathcal{S}(\lambda_0)\leq \mathcal{Q}_{\lambda_0}(u) < \mathcal{S}(\lambda_0)+\frac{\e}{2}.$$
  From this, using the monotonicity of $\mathcal{S}(\cdot)$ and exploiting
  H\"older's inequality 
  (recall that $u\in\mathcal{V}(\Omega)$ and $u\equiv 0$ a.e.\,in $\R^n\setminus\Omega$),
  for every $\lambda < \lambda_0$ we obtain
  \begin{align*}
   & 0 \leq \mathcal{S}(\lambda)-\mathcal{S}(\lambda_0)
   \leq \mathcal{Q}_{\lambda}(u)-\mathcal{S}(\lambda_0)\\
   & \qquad\quad = \big(\mathcal{Q}_{\lambda_0}(u)-\mathcal{S}(\lambda_0)\big)
   + (\lambda_0-\lambda)\|u\|^2_{L^2(\R^n)} \\
   & \qquad\quad < \frac{\e}{2}+(\lambda_0-\lambda)\|u\|^2_{L^2(\R^n)} \\
   & \qquad\quad \leq \frac{\e}{2}+(\lambda_0-\lambda)|\Omega|^{4/n}
  \end{align*}
  As a consequence, setting $\delta_\e := \e/(2|\Omega|^{4/n})$, we conclude that
  $$0 \leq \mathcal{S}(\lambda)-\mathcal{S}(\lambda_0) < 
  \e\qquad\text{for every $\lambda_0-\delta_\e
  < \lambda \leq \lambda_0$},$$
  and this proves that $\mathcal{S}(\cdot)$  is continuous from the left at $\lambda_0$.
  \vspace*{0.05cm}
  
  As regards the continuity from the right, we proceed {essentially as above}:
  first of all, given any $\e > 0$ and any $\lambda > \lambda_0$, we can find
  $u = u_{\e,\lambda}\in\mathcal{V}(\Omega)$ such that
  $$\mathcal{S}(\lambda)\leq \mathcal{Q}_{\lambda}(u) < \mathcal{S}(\lambda)+\frac{\e}{2}.$$
  From this, by the monotonicity of $\mathcal{S}(\cdot)$ and
  H\"older's inequality, we obtain 
  \begin{align*}
   & 0 \leq \mathcal{S}(\lambda_0)-\mathcal{S}(\lambda)
   \leq \mathcal{Q}_{\lambda_0}(u)-\mathcal{S}(\lambda) \\
   & \qquad\quad = \big(\mathcal{Q}_{\lambda}(u)-\mathcal{S}(\lambda)\big)
   + (\lambda-\lambda_0)\|u\|^2_{L^2(\R^n)} \\
   & \qquad\quad \leq \frac{\e}{2}+(\lambda-\lambda_0)|\Omega|^{4/n} < \e,
  \end{align*}
  provided that $\lambda_0 \leq \lambda < \lambda_0+\delta_\e$ (where $\delta_\e > 0$ is as above). 
    We then conclude that $\mathcal{S}(\cdot)$ 
  is also continuous from the right at $\lambda_0$,
  and the proof is complete.  
 \end{proof}
In the light of the results established so far, we can finally provide the
 \begin{proof}[Proof of Theorem \ref{thm:mainLinear}] 
  To begin with, we define
  $$\lambda^* := \sup\big\{\lambda > 0:\,\mathcal{S}(\mu) = 
  \mathcal{S}_n\,\,\text{for all $0<\mu\leq\lambda$}\big\}.$$
  On account of Lemma \ref{lem:SlambdaequalS}, we see that $\lambda_{1,s}\leq \lambda^* < \infty$;
  moreover, since we know from Lemma \ref{lem:contSnlambda} that $\mathcal{S}(\cdot)$
  is continuous, we have
  $\mathcal{S}(\lambda^*) = \mathcal{S}_n.$ 
  
  We also notice that, 
  since $\mathcal{S}(\cdot)\geq 0$ on $(0,\lambda_1]$ and $\mathcal{S}(\cdot) < 0$ on $(\lambda_1,\infty)$,
  again by the con\-ti\-nui\-ty of $\mathcal{S}(\cdot)$ we infer that $\mathcal{S}(\lambda_1) = 0$;
  as a consequence, recalling that $\mathcal{S}(\cdot)$ is nonincreasing on $(0,\infty)$,
  we conclude that
  $$\lambda^*\in [\lambda_{1,s},\lambda_1).$$
  We now turn to prove that the assertion of Theorem \ref{thm:mainLinear} 
  holds \emph{with this choice of $\lambda^*$}, that is,
  we show that the following assertions hold:
  \begin{itemize}
   \item[(i)]  pro\-blem \eqref{eq:mainPBlinear} possesses at least one
  solution if 
    $\lambda^*<\lambda<\lambda_1$;
    \vspace*{0.1cm}
    
    \item[(ii)] there do not exist solutions to problem \eqref{eq:mainPBlinear} 
    if $\lambda \geq \lambda_1$;
    \vspace*{0.1cm}
    
    \item[(iii)] for every fixed $0<\lambda\leq \lambda_{1,s}$,
    there do not exist solutions to problem \eqref{eq:mainPBlinear} 
    belonging to the closed ball
    $$\mathcal{B} = \{u\in L^{2^*}(\Omega):\,\|u\|\leq\mathcal{S}_n^{\,(n-2)/4}\}.$$
  \end{itemize}
  To this end, we treat the following three cases separately.
  \medskip

 \textsc{Case I: $\lambda^*<\lambda < \lambda_{1}$}. In this case, owing to the definition of $\lambda^*$ (and recalling that $\mathcal{S}(\cdot)$ is continuous), we have
 $0\leq \mathcal{S}(\lambda) < \mathcal{S}_n$; from this,
 by arguing exactly as in the proof of \cite[Lem.\,1.2]{BrNir}, 
 we see that $\mathcal{S}(\lambda)$ \emph{is achieved}, that is,
 there exists a nonzero function $w\in\mathcal{V}(\Omega)$ such that
 $$\mathcal{Q}_\lambda(w) = \mathcal{S}(\lambda).$$
 In particular, since $\lambda < \lambda_1$ (and $w\not\equiv 0$), we obtain
 \begin{align*}
  \mathcal{S}(\lambda) & = \|w\|^2_{L^2(\R^n)}\bigg[\frac{\rho(w)^2}{\|w\|^2_{L^2(\R^n)}}-\lambda\bigg]
  \geq \|w\|^2_{L^2(\R^n)}(\lambda_1-\lambda) > 0.
 \end{align*}  
 Now we have proved that $\mathcal{S}(\lambda)$ is strictly positive
 and is achieved, we are then entitled to apply Lemma \ref{lem:solMinLinear},
 ensuring that there exists a solution to \eqref{eq:mainPBlinear}.
 \medskip
 
 \textsc{Case II: $\lambda \geq \lambda_{1}$}. In this second case, since
 $\mathcal{S}(\cdot)$ is nonincreasing (and since we know that $\mathcal{S}(\lambda_1) = 0$), we infer that
 $$\mathcal{S}(\lambda)\leq 0 < \mathcal{S}_n.$$
 Owing to this fact, 
 we can argue once again as in \cite[Lem.\,1.2]{BrNir} to prove that $\mathcal{S}(\lambda)$ is achieved;
 however, since $\mathcal{S}(\lambda)\leq 0$, we \emph{are not entitled} to 
 apply Lemma \ref{lem:SlambdaequalS}
 to derive the existence of a solution to problem \eqref{eq:mainPBlinear}.
  Rather, we show that \emph{there cannot exist solutions} to
 \eqref{eq:mainPBlinear} by
  proceeding as in \cite[Remark~1.1]{BrNir}.
  
  Arguing by contradiction, let us assume that there exists (at least)
  one solution $u$ to problem \eqref{eq:mainPBlinear}
  in this case. On account of Remark \ref{rem:problambdalambdas}, we know that
  exists a nonzero function $\psi_0\in\mathcal{X}^{1,2}(\Omega)$ such that
  $\psi_0 > 0$ a.e.\,in $\Omega$ and
  \begin{align*}
  & \int_{\Omega}\nabla \psi_0\cdot\nabla v\,dx 
  + \iint_{\R^{2n}}\frac{(\psi_0(x)-\psi_0(y))(v(x)-v(y))}{|x-y|^{n+2s}}\,dx\,dy \\
  & \qquad\quad = \lambda_{1}\int_{\Omega}\psi_0v\,d x
  \qquad \forall\,\,v\in\mathcal{X}^{1,2}(\Omega),
  \end{align*}
  that is, $\psi_0$ is an \emph{eigefunction} for $\LL$ associated with $\lambda_1$.
  
  In particular, choosing $v = u\in\mathcal{X}^{1,2}_+(\Omega)$ in the above identity
  and
  recalling
  that $u$ is a solution to problem \eqref{eq:mainPBlinear}
  (according to Definition \ref{def:solPb}), 
  we obtain
  \begin{align*}
   \lambda_{1}\int_{\Omega}\psi_0u\,dx
   & = \int_{\Omega}\nabla \psi_0\cdot\nabla u\,dx 
  + \iint_{\R^{2n}}\frac{(\psi_0(x)-\psi_0(y))(u(x)-u(y))}{|x-y|^{n+2s}}\,dx\,dy \\
  & = \int_{\Omega}(u^{2^*-1}+\lambda u)\psi_0\,d x \\
  & (\text{since $u,\psi_0 > 0$ a.e.\,in $\Omega$, see Theorem \ref{thm:aprioriPropSol}}) \\
  & > \lambda\int_{\Omega}u\psi_0\,d x,
  \end{align*}
  but this is clearly in contradiction with the fact that $\lambda\geq \lambda_1$.
  \medskip
  
  \textsc{Case III: $0<\lambda\leq \lambda_{1,s}$}. 
  Arguing by contradiction, let us suppose that there exists a solution $u$ to
  problem \eqref{eq:mainPBlinear} such that $u\in\mathcal{B}$. 
  Then, setting
  $$v := u/\|u\|_{L^{2^*}(\Omega)},$$
  we have the following computation, based on the properties of $u$:
  \begin{align*}
   \mathcal{Q}_\lambda(v) & = \frac{1}{\|u\|^2_{L^{2^*}(\Omega)}}\mathcal{Q}_\lambda(u) 
    = \frac{1}{\|u\|^2_{L^{2^*}(\Omega)}}\big(\rho(u)^2-\lambda\|u\|^2_{L^2(\Omega)}\big) \\
    & = \frac{1}{\|u\|^2_{L^{2^*}(\Omega)}}
    \bigg\{\int_{\Omega}|\nabla u|^2\,d x+\iint_{\R^{2n}}
    \frac{|u(x)-u(y)|^2}{|x-y|^{n+2s}}\,dx\,dy-\lambda\int_\Omega u^2\,d x\bigg\} \\
    &  =
    \frac{1}{\|u\|^2_{L^{2^*}(\Omega)}}\int_\Omega u^{2^*}\,dx \qquad
    (\text{since $u$ solves \eqref{eq:mainPBlinear}, see \eqref{eq:weakSol} with $v = u$}) \\[0.1cm]
     & = \|u\|^{2^*-2}_{L^{2^*}(\Omega)} 
   \leq \mathcal{S}_n \qquad (\text{since $u\in\mathcal{B}$}).
    \end{align*}
  As a consequence, owing to Lemma \ref{lem:SlambdaequalS}, we obtain
  \begin{equation} \label{eq:SnachievedCaseI}
   \mathcal{Q}_\lambda(v) = \mathcal{S}_n,
  \end{equation}
  and this proves that
  $\mathcal{S}(\lambda) = \mathcal{S}_n$
  \emph{is achieved} by $v$.
  \vspace*{0.1cm}
  
  Now, since $v\in\mathcal{X}^{1,2}(\Omega)\subseteq
 \mathcal{D}_0^{1,2}(\Omega)$, from \eqref{eq:SnachievedCaseI} we get
 \begin{align*}
  \mathcal{S}_n & \leq \|\nabla v\|^2_{L^2(\R^n) }
  = \mathcal{Q}_\lambda(v) -
  \big([v]^2_s-\lambda\|v\|^2_{L^2(\R^n)}\big) \\
  & (\text{since $\lambda\leq \lambda_{1,s}$ and $v\in \mathcal{X}^{1,2}(\Omega)\subseteq
  \mathcal{D}^{s,2}_0(\Omega)$}) \\
  & \leq \mathcal{Q}_\lambda(v) = \mathcal{S}_n,
 \end{align*}
 and this shows that 
 the best Sobolev constant $\mathcal{S}_n$ is achieved by $v\in \mathcal{D}_0^{1,2}(\Omega)$.

 On the other hand, since $\Omega$ is \emph{bounded}, we know 
 from Remark \ref{rem:propSnrecall} that $\mathcal{S}_n$ 
 \emph{is never achieved} in $\mathcal{X}^{1,2}(\Omega)$, and so we have a contradiction.
 \medskip
 
 \noindent This ends the proof.
 \end{proof}
 
 \subsection{The superlinear case \texorpdfstring{$1<p < 2^*-1$}{1<p<2^*-1}}  \label{sec:superlinear}
Now we continue the study of the solvability of \eqref{eq:mainPBBN} in the case $1<p<2^{*}-1$,
and we address the proof of Theorem~\ref{MS:3232}. 
Compared to the {\em linear} case previously treated, 
the main difference is that we \emph{cannot} link the solvability
of problem \eqref{eq:mainPBBN} to the existence of constrained
\emph{minimizers} for the quadratic form $\mathcal{Q}_\lambda$
 in \eqref{eq:QformQlambda}
(since the rescaling argument in the proof of Lemma \ref{lem:solMinLinear} cannot be applied when
$p > 1$); as a consequence, we may try to prove the existence
of solutions to \eqref{eq:mainPBBN} by showing the existence
of unconstrained and nonnegative \emph{critical points} for the functional $\mathcal{F}_{\lambda,p}$ in \eqref{eq:defFlambda}.

However, this approach presents a drawback:
even if we are able to prove the existence of a nonzero critical point $v$ for $\mathcal{F}_{\lambda,p}$,
there is no reason for $u := |v|$ to be a critical point, and so we cannot be sure that 
there exist \emph{nonnegative} critical points.
For this reason, and following the approach in
\cite{BrNir}, 
we consider a slightly modified functional with respect to~\eqref{eq:defFlambda}, namely,
  \begin{equation} \label{eq:defJlambda}
 \mathcal{J}_{\lambda, p}(u) := \frac{1}{2}\rho(u)^2
 -\frac{1}{2^*}\int_{\Omega}(u_+)^{2^*}\,dx-\frac{\lambda}{p+1}\int_\Omega(u_+)^{p+1}\,d x
 \quad (u\in\mathcal{X}^{1,2}(\Omega)),
 \end{equation}
 \noindent where $u_+:= \max\{u,0\}$ denotes the positive part of the function $u$.
 
 Now, it is easy to see that any (nonzero) critical point
 of $\mathcal{J}_{\lambda, p}$ is a solution to \eqref{eq:mainPBBN}:
 indeed, if $u\not\equiv 0$ is a critical point of this functional, 
 by writing down the associated Euler-Lagrange equation
 we see that  $u$ is a weak solution 
 of
 $$\begin{cases}
  \LL u = (u_+)^{2^*-1}+\lambda(u_+)^{p-1} & \text{in $\Omega$} ,\\
  u \equiv 0 & \text{in $\R^n\setminus\Omega$}.
 \end{cases}$$
In particular, since we are assuming $\lambda > 0$, we deduce that $\LL u \geq 0$ in
 the weak sense in $\Omega$ (see, precisely, \eqref{eq:LLugeqzeroweak} in the proof
 of Theorem \ref{thm:aprioriPropSol}). 
 This, together with the fact that $u\equiv 0$ a.e.\,in $\R^n\setminus\Omega$,
 allows us to apply the Weak Maximum Principle in \cite[Theorem 1.2]{BDVV22a}, ensuring that
 $u\geq 0$ a.e.\,in $\R^n$.
 We then conclude that 
 $$u_+ \equiv u,$$ 
 and thus $u$ is a solution to problem \eqref{eq:mainPBBN}
 (according to Definition \ref{def:solPb}).
 \medskip
 
 In view of this fact, to prove Theorem \ref{MS:3232}
 it then suffices to prove the existence of a nonzero
 critical point for
 $\mathcal{J}_{\lambda,p}$.
 
 As in the purely local (see \cite{BrNir}) or purely nonlocal case (see \cite{BCSS}), 
 the main difficulty in the application of the Mountain Pass Theorem consists in proving the validity of a 
 Palais--Smale  condition at level $c\in \mathbb{R}$, usually denoted by $(PS)_c$.  
 In particular, we have to prove that the Mountain Pass
 critical level of $\mathcal{J}_{\lambda, p}$ lies below the threshold of application of the $(PS)_c$ condition. 
 In doing this, we will understand better the {\em strange behaviour in the linear case} and we will notice how 
 the local part of $\mathcal{L}$ affects the problem.
 \medskip
 
 To this end, the first step is to 
note that the functional $\mathcal{J}_{\lambda, p}$ has a nice Mountain Pass geometry,
 as described by the following lemma.
\begin{lemma}\label{prop:geometry}
There exist two positive constants $\alpha$, $\beta>0$ such that
\begin{itemize}
\item[i)] for any $u \in \mathcal{X}^{1,2}(\Omega)$ with $\rho(u)=\alpha$, it holds that $\mathcal{J}_{\lambda, p}(u)\geq \beta$;
\item[ii)] there exists a positive function $e\in \mathcal{X}^{1,2}(\Omega)$ such that 
$$\text{$\rho(e)>\alpha$ and $\mathcal{J}_{\lambda, p}(e)<\beta$}.$$
\end{itemize}
Moreover, for every positive function $u \in \mathcal{X}^{1,2}(\Omega)$, it holds that
\begin{equation}
\lim_{t\to 0^+}\mathcal{J}_{\lambda, p}(tu) =0.
\end{equation}
\begin{proof}
The proof is quite standard and exploits the validity of a Sobolev embedding for the space $\mathcal{X}^{1,2}(\Omega)$. See e.g. \cite[Proposition 3.1]{BCSS}.
\end{proof}
\end{lemma}
 
 We now move to show that $\mathcal{J}_{\lambda, p}$ satisfies a local Palais-Smale condition at a level $c\in\mathbb{R}$ related to the best Sobolev constant $\mathcal{S}_n$.
 
 \begin{lemma}\label{prop:PalaisSmale}
 The functional $\mathcal{J}_{\lambda, p}$ satisfies the $(PS)_c$ for every $c < \dfrac{1}{n}(\mathcal{S}_n)^{n/2}$.
 \begin{proof}
 Let $\{u_m\}$ be a $(PS)_c$ sequence for the functional $\mathcal{J}_{\lambda,p}$ in $\mathcal{X}^{1,2}(\Omega)$, i.e.
 \begin{itemize}
 \item[i)] $\mathcal{J}_{\lambda,p}(u_m) \to c$, as $m\to +\infty$;
 \item[ii)]$\mathcal{J}'_{\lambda,p}(u_m) \to 0$, as $m\to +\infty$.
 \end{itemize}
 By standard arguments, this ensures that $\{u_m\}$ is bounded in the Hilbert space $\mathcal{X}^{1,2}(\Omega)$. In particular, this implies the existence of a function $u_{\infty}\in \mathcal{X}^{1,2}(\Omega)$ such that, up to subsequences,
 \begin{equation}
 \begin{aligned}
 \int_{\Omega}\langle &\nabla u_m, \nabla \varphi \rangle \, dx + \iint_{\mathbb{R}^{2n}}\dfrac{(u_m(x)-u_m(y))(\varphi(x)-\varphi(y))}{|x-y|^{n+2s}}\, dx dy \\
 &\to  \int_{\Omega}\langle \nabla u_{\infty}, \nabla \varphi \rangle \, dx + \iint_{\mathbb{R}^{2n}}\dfrac{(u_{\infty}(x)-u_{\infty}(y))(\varphi(x)-\varphi(y))}{|x-y|^{n+2s}}\, dx dy,
 \end{aligned}
 \end{equation}
 \noindent for every  $\varphi \in \mathcal{X}^{1,2}(\Omega)$.
Moreover, $u_m$ converges to $u_{\infty}$ (as $m\to +\infty$) weakly in $L^{2^*}(\Omega)$, almost everywhere and strongly in $L^{r}(\Omega)$ for every $r \in [1,2^*)$ (and hence in norm).
Putting everything together, we get that $u_{\infty}$ is a weak solution to \eqref{eq:mainPBBN} (with $p>1$).

Now, by exploiting \cite[Theorem 1]{BrLi}, we find that
\begin{equation}
\begin{aligned}
\mathcal{J}_{\lambda,p}(u_m) &= \dfrac{1}{2}\int_{\Omega}|\nabla u_m|^2 + \dfrac{1}{2}\iint_{\mathbb{R}^{2n}}\dfrac{|u_m(x)-u_m(y)|^2}{|x-y|^{n+2s}}\, dx dy \\
&\quad - \dfrac{1}{2^*}\int_{\Omega}((u_m)_+)^{2^*}\, dx - \dfrac{\lambda}{p+1}\int_{\Omega}((u_m)_+)^{p+1}\, dx\\
&= \mathcal{J}_{\lambda,p}(u_{\infty}) + \dfrac{1}{2}\rho(u_m -u_{\infty})^2 - \dfrac{1}{2^*}\int_{\Omega}|(u_m)_{+} -(u_{\infty})_{+} |^{2^{*}}\, dx + o(1).
\end{aligned}
\end{equation}
Arguing as in the proof of Claim 3 of \cite[Proposition 3.2]{BCSS}, we can further prove that
\begin{equation}\label{eq:Claim3}
\begin{aligned}
\rho(u_m - u_{\infty})^2 &= \int_{\Omega}|(u_m)_+ - (u_{\infty})_+|^{2^*}\, dx + o(1)\\
&\leq \int_{\Omega}|u_m - u_{\infty}|^{2^*}\, dx + o(1).
\end{aligned}
\end{equation}
We are now ready to finish the proof of Lemma~\ref{prop:PalaisSmale}. By \eqref{eq:Claim3}, we have that
\begin{equation}
\dfrac{1}{2}\rho(u_m - u_{\infty})^2 - \dfrac{1}{2^*}\int_{\Omega}|(u_m)_+ - (u_{\infty})_+|^{2^*}\, dx = \dfrac{1}{n}\rho(u_m - u_{\infty})^2 + o(1),
 \end{equation}
 \noindent and therefore
 \begin{equation}
 \mathcal{J}_{\lambda,p}(u_{\infty}) + \dfrac{1}{n}\rho(u_m - u_{\infty})^2 = \mathcal{J}_{\lambda,p}(u_m) + o(1) = c_2 + o(1),
 \end{equation}
 \noindent as $m\to +\infty$.
 On the other hand, thanks to the boundedness of $\{u_m\}$, possibily passing to a subsequence, we have that there exists $L>0$ such that
 \begin{equation}
 \rho(u_m - u_{\infty})^2 \to L, \quad \textrm{ as } m \to +\infty,
 \end{equation}
 \noindent and that there exists $\tilde{L}\geq L$ such that
 \begin{equation}
 \int_{\Omega}|u_m - u_{\infty}|^{2^*}\, dx \to \tilde{L}.
 \end{equation}
 Recalling that 
 $$\mathcal{S}_{n,s}(\Omega) = \inf \dfrac{\rho(u)^2}{\|u\|^{2}_{L^{2^*}(\Omega)}} = \mathcal{S}_n,$$
 \noindent we find that 
 \begin{equation}
 L \geq \mathcal{S}_{n,s}(\Omega) \tilde{L}^{2/2^*} = \mathcal{S}_n \tilde{L}^{2/2^*},
 \end{equation}
 \noindent which implies that either $L=0$ or $L \geq \mathcal{S}_n \tilde{L}^{n/2}$. Arguing as in \cite[Proposition 3.2]{BCSS} we can show that the only possible case is $L=0$, and this shows that the sequence $\{u_m\}$ strongly converges in $\mathcal{X}^{1,2}(\Omega)$.
 \end{proof}
 \end{lemma}
 
The last ingredient needed to establish Theorem~\ref{MS:3232}
consists in showing the existence of a path whose energy is below the critical threshold $\mathcal{S}_n/n$. 
To this aim, we first introduce an auxiliary function reminiscent of the one used in \cite{BrNir}: let $\phi_0 \in C^{\infty}(\mathbb{R})$ be a nonincreasing cut--off function given by
\begin{equation}
\phi_0 (t) := \left\{ \begin{array}{rl}
1, &{\mbox{ if }} t\in \left[0, \tfrac{1}{4}\right],\\[0.15cm]
0, &{\mbox{ if }} t\geq \frac{1}{2}.
\end{array}\right.
\end{equation}   
Then, we define the function $\phi_r(x):\mathbb{R}^n \to \mathbb{R}$ as 
 \begin{equation}
 \phi_r(x):= \phi_0 \left(\dfrac{|x|}{r}\right), 
 \end{equation}
 \noindent for a given $r>0$ such that $\overline{B_r(0)}\subset \Omega$.
  Finally, for all~$\varepsilon>0$, let
 \begin{equation*}
 U_{\varepsilon}(x):= \dfrac{\varepsilon^{(n-2)/2}}{\left(|x|^2+\varepsilon^2 \right)^{(n-2)/2}} \qquad
 \text{and}\qquad \eta_{\varepsilon}(x) := \dfrac{\phi_r(x)\, U_{\varepsilon}(x)}{\|\phi_r\, U_{\varepsilon}\|_{L^{2^*}(\Omega)}}.
 \end{equation*}
 In this framework, we have:
 \begin{lemma}\label{lem:Path}
 Let $n\geq 3$ and $p\in(1,2^*-1)$. Moreover, let
 $\kappa_{s,n},\,\beta_{p,n}$ be as
 in \eqref{eq:defkappabeta}.
 Then, the following assertions hold.
 \begin{itemize}
  \item[(i)] If $\kappa_{s,n} > \beta_{p,n}$, then there exists $\varepsilon>0$ small enough such that 
 \begin{equation*}
 \sup_{t \geq 0}\mathcal{J}_{\lambda,p}(t\eta_{\varepsilon}) < \dfrac{1}{n}(\mathcal{S}_n)^{n/2}\quad
 \text{for all $\lambda > 0$}.
 \end{equation*}
 \item[(ii)] If, instead, $\kappa_{s,n} \leq 
 \beta_{p,n}$, then there exist $\e > 0$ and $\lambda_0 > 0$ such that
  $$\sup_{t \geq 0}\mathcal{J}_{\lambda,p}(t\eta_{\varepsilon}) < \dfrac{1}{n}(\mathcal{S}_n)^{n/2}\quad
 \text{for all $\lambda \geq \lambda_0$}.$$
\end{itemize}
 \end{lemma}
\begin{proof}
 We closely follow the argument in \cite[Lemma 3.4]{BCSS}, which in turn is a mo\-di\-fi\-ca\-tion 
 of the original approach by Brezis and Nirenberg \cite{BrNir}. 
 
 To begin with we observe that, by definition, we have
 \begin{equation} \label{eq:Jtoestimate}
 \begin{split}
\mathcal{J}_{\lambda,p}(t \eta_{\varepsilon}) & = \dfrac{t^2}{2}\int_{\Omega}|\nabla \eta_{\varepsilon}|^2\, dx +   
   \dfrac{t^2}{2}\iint_{\mathbb{R}^{2n}}\dfrac{|\eta_{\varepsilon}(x)-\eta_{\varepsilon}(y)|^2}{|x-y|^{n+2s}}\, dx dy \\
   & \qquad\qquad - \dfrac{t^{2^*}}{2^{*}}-\lambda \dfrac{t^{p+1}}{p+1}\int_{\Omega}\eta_\e(x)^{p+1}\, dx;
\end{split}
\end{equation}
we then turn to estimate the three integrals in the right-hand side of \eqref{eq:Jtoestimate}.
As regards the $L^{p+1}$-norm of $\eta_\e$, if $\e > 0$ is small enough we have
\begin{equation}\label{eq:etalower}
\begin{aligned}
\int_{\mathbb{R}^n}\eta_{\varepsilon}(x)^{p+1}\, dx &= C \int_{B_{r}(0)}U_{\varepsilon}(x)^{p+1}\, dx \\
& = C \, \varepsilon^{-(p+1)\tfrac{n-2}{2}}\int_{0}^{r}\dfrac{\tau^{n-1}}{\left(\tfrac{\tau^2}{\varepsilon^2}+1 \right)^{(p+1)\tfrac{n-2}{2}}}\, d\tau\\
&= C \, \varepsilon^{n-(p+1)\tfrac{n-2}{2}}\int_{0}^{r/\varepsilon}\dfrac{t^{n-1}}{\left(t^2+1 \right)^{(p+1)\tfrac{n-2}{2}}}\, dt \\
&\geq C \, \varepsilon^{n-(p+1)\tfrac{n-2}{2}}\int_{0}^{1}\dfrac{t^{n-1}}{\left(t^2+1 \right)^{(p+1)\tfrac{n-2}{2}}}\, dt \\
&= C\varepsilon^{n-(p+1)\tfrac{n-2}{2}},
\end{aligned}
\end{equation}
where the constant $C > 0$ is adjusted step by step, but it depends only on $n$.
As re\-gards the $L^2$-norm of $\nabla\eta_\e$, instead, we know 
from \cite[Lemma 1.1]{BrNir} that
$$\int_{\Omega}|\nabla \eta_{\varepsilon}|^2\, dx = \mathcal{S}_{n} + O(\varepsilon^{n-2})\qquad
\text{as $\e\to 0^+$}.$$
We are then left to estimate the Gagliardo seminorm of $\eta_\e$.
To begin with, arguing as in \cite[Proposition 21]{SerValTAMS}
(and recalling that $B_r(0)\subset\Omega$), we define
\begin{align*}
& \mathbb{D}:= \left\{(x,y)\in \mathbb{R}^{2n}: x \in B_{r/4}(0), y\in B^c_{r/4}(0) \, 
\textrm{ and } |x-y|>r/8 \right\}, \\
& \mathbb{E}:= \left\{(x,y)\in \mathbb{R}^{2n}: x \in B_{r/4}(0), y\in B^c_{r/4}(0) \,
 \textrm{ and } |x-y|\leq r/8\right\}.
\end{align*}
With this notation at hand, we get that
\begin{equation}
\begin{aligned}
& \iint_{\mathbb{R}^{2n}}\dfrac{|\eta_{\varepsilon}(x)-\eta_{\varepsilon}(y)|^2}{|x-y|^{n+2s}}\, dx dy \\
& \qquad = \dfrac{1}{\|\phi_r U_{\varepsilon}\|^2_{L^{2^{\ast}}(\Omega)}} \iint_{B_{r/4}(0)\times B_{r/4}(0)}\dfrac{|U_{\varepsilon}(x)-U_{\varepsilon}(y)|^2}{|x-y|^{n+2s}}\, dx dy \\
&\qquad\quad+ 2  \iint_{\mathbb{D}}\dfrac{|\eta_{\varepsilon}(x)-\eta_{\varepsilon}(y)|^2}{|x-y|^{n+2s}}\, dx dy\\
&\qquad\quad+ 2  \iint_{\mathbb{E}}\dfrac{|\eta_{\varepsilon}(x)-\eta_{\varepsilon}(y)|^2}{|x-y|^{n+2s}}\, dx dy\\
&\qquad\quad+  \iint_{B^c_{r/4}(0)\times B^c_{r/4}(0)}\dfrac{|\eta_{\varepsilon}(x)-\eta_{\varepsilon}(y)|^2}{|x-y|^{n+2s}}\, dx dy.
\end{aligned}
\end{equation}
Now, by repeating the proof of \cite[Proposition 21]{SerValTAMS}, one realizes that the last three integrals 
in the above identity 
behave like $O(\varepsilon^{n-2})$ as $\e\to 0^+$.
Moreover, we claim that the first integral behaves like $O(\varepsilon^{2-2s})$ as $\e\to 0^+$. 

Indeed, it is proved in \cite{BrNir, BN2} that
\begin{equation*}
\|\phi_r U_{\varepsilon}\|^{2^{\ast}}_{L^{2^{\ast}}(\Omega)} = C_n + O(\varepsilon^{n}) \quad \textrm{as } \e\to 0^+,
\end{equation*}
\noindent where $C_n > 0$ is a suitable constant only depending on $n$; on the other hand,
by performing the change of variables $x = \varepsilon \xi$ and~$y= \varepsilon \zeta$, we get
\begin{equation*}
\begin{split}
& \iint_{B_{r/4}(0)\times B_{r/4}(0)}\dfrac{|U_{\varepsilon}(x)-U_{\varepsilon}(y)|^2}{|x-y|^{n+2s}}\, dx dy 
\leq \int_{\R^{2n}}\dfrac{|U_{\varepsilon}(x)-U_{\varepsilon}(y)|^2}{|x-y|^{n+2s}}\, dx dy \\
& \qquad = 
\int_{\R^{2n}}\frac{\e^{n-2}}{\e^{(2n-4)+n+2s}}\cdot\dfrac{|U_{1}(\xi)-U_{1}(\zeta)|^2}
{|\xi-\zeta|^{n+2s}}\cdot\e^{2n}\, dx dy \\[0.1cm]
& \qquad = \e^{2-2s}\cdot[U_1]^2_{s} = 
O(\varepsilon^{2-2s}),
\end{split}
\end{equation*}
where we have used Remark \ref{rem:Snasymptotic} (notice that $U_1 = c^{-1}\,\mathcal{U}$).
\vspace*{0.1cm}

 Gathering all these information, from \eqref{eq:Jtoestimate} we then obtain
\begin{equation} \label{eq:tociteIntro}
\begin{aligned}
\mathcal{J}_{\lambda,p}(t \eta_{\varepsilon})&\leq \dfrac{t^2}{2}(\mathcal{S}_n+O(\varepsilon^{n-2})+O(\varepsilon^{2-2s})) - \dfrac{t^{2^*}}{2^{*}} - C \lambda \dfrac{t^{p+1}}{p+1}\varepsilon^{n-(p+1)\tfrac{n-2}{2}}\\
&= \dfrac{t^2}{2}\big(\mathcal{S}_n+ O(\varepsilon^{\kappa_{s,n}})\big) 
 - \dfrac{t^{2^*}}{2^{*}} - C \lambda \dfrac{t^{p+1}}{p+1}\varepsilon^{\beta_{p,n}} \\
 & \leq 
 \dfrac{t^2}{2}\big(\mathcal{S}_n+ C\varepsilon^{\kappa_{s,n}}\big) 
 - \dfrac{t^{2^*}}{2^{*}} - C \lambda \dfrac{t^{p+1}}{p+1}\varepsilon^{\beta_{p,n}}=: g(t),
\end{aligned}
\end{equation}
provided that $\e > 0$ is sufficiently small and for a suitable constant $C > 0$.
\vspace*{0.1cm}

To proceed further, following the proof of \cite[Lemma 3.4]{BCSS}, we turn to study
the maximum/minimum points of $g$. To this end we first observe that, since
$$\text{$g(t)\to -\infty$ as $t\to\infty$},$$ 
there exists some point $t_{\varepsilon,\lambda}\geq 0$ such that 
$$\sup_{t \geq 0}g(t) = g(t_{\varepsilon,\lambda}).$$
If $t_{\e,\lambda} = 0$, we have $g(t)\leq 0$ for all $t\geq 0$, and the lemma 
is trivially established as a consequence 
of \eqref{eq:tociteIntro}. 
If, instead, $t_{\e,\lambda} > 0$, since $g\in C^1((0,\infty))$ we get
\begin{equation*}
0= g'(t_{\varepsilon,\lambda}) = t_{\varepsilon,\lambda}
(\mathcal{S}_n + C \varepsilon^{\kappa_{s,n}})- t_{\varepsilon,\lambda}^{2^{*} -1}-
C\lambda t_{\varepsilon,\lambda}^{p}\varepsilon^{\beta_{p,n}}.
\end{equation*}
In particular, recalling that $t_{\varepsilon,\lambda}>0$, we can rewrite the above identity as
\begin{equation} \label{eq:gprimezerosemp}
 \mathcal{S}_n + C \varepsilon^{\kappa_{s,n}} = 
t_{\varepsilon,\lambda}^{2^{*} -2}+C\lambda t_{\varepsilon,\lambda}^{p-1}\varepsilon^{\beta_{p,n}}
\end{equation}
from which we easily derive that
$$t_{\varepsilon,\lambda}< ( \mathcal{S}_n + C \varepsilon^{\kappa_{s,n}})^{1/(2^{*}-2)}.$$
We now distinguish two cases, according to the assumptions.
\medskip

\textsc{Case (i): $\kappa_{s,n} > \beta_{p,n}$}. In this case we first observe that, since 
$t_{\e,\lambda} > 0$, from identity 
\eqref{eq:gprimezerosemp} we easily infer the existence
of some $\mu_\lambda > 0$ such that
$$t_{\e,\lambda}\geq \mu_\lambda > 0\quad\text{provided that $\e$ is small enough}.$$
This, together with the fact that since the map
$$t \mapsto \dfrac{t^2}{2}( \mathcal{S}_n +C \varepsilon^{\kappa_{s,n}}) - \dfrac{t^{2^*}}{2^*}$$
\noindent is increasing in the closed interval
 $[0,\mathcal{S}_n + C \varepsilon^{\kappa_{s,n}})^{1/(2^{*}-2)}]$, implies
\begin{equation*}
\begin{aligned}
& \sup_{t\geq 0}g(t) =g(t_{\varepsilon,\lambda}) \\[0.1cm]
& \qquad < \dfrac{(\mathcal{S}_n + C \varepsilon^{\kappa_{s,n}})^{1+ 2/(2^{*}-2)}}{2} -  
 \dfrac{(\mathcal{S}_n + C \varepsilon^{\kappa_{s,n}})^{2^*/(2^{*}-2)}}{2^*} - 
 C \varepsilon^{\beta_{p,n}}\\[0.1cm]
& \qquad = \dfrac{1}{n}(\mathcal{S}_n +C \varepsilon^{\kappa_{s,n}})^{n/2}-
C\varepsilon^{\beta_{p,n}}
\\[0.1cm]
& \qquad \leq \dfrac{1}{n}(\mathcal{S}_n)^{n/2}+ C \varepsilon^{\kappa_{s,n}} - C\varepsilon^{\beta_{p,n}}
 <\dfrac{1}{n}(\mathcal{S}_n)^{n/2},
\end{aligned}
\end{equation*}
\noindent provided that $\e > 0$ is sufficiently small. 
We explicitly stress that, in the last estimate, we have exploited in a crucial way
the assumption $\kappa_{s,n} > \beta_{p,n}$.
\medskip

\textsc{Case (ii):  $\kappa_{s,n} \leq \beta_{p,n}$.} In this second case, we 
begin by claiming that
\begin{equation} \label{eq:claimtlambdazero}
 \lim_{\lambda\to\infty}t_{\e,\lambda} = 0.
\end{equation}
Indeed, suppose by contradiction that $\ell: = \limsup_{\lambda\to\infty}t_{\e,\lambda} > 0$: 
then, by possibly choosing a sequence $\{\lambda_k\}_{k}$ diverging to $\infty$, 
from \eqref{eq:gprimezerosemp} we get
\begin{align*}
 \mathcal{S}_n + C \varepsilon^{\kappa_{s,n}} = 
t_{\varepsilon,\lambda_k}^{2^{*} -2}+C\lambda_k t_{\varepsilon,\lambda_k}^{p-1}
\varepsilon^{\beta_{p,n}}
\to\infty
\end{align*}
but this is clearly absurd. Now we have established 
\eqref{eq:claimtlambdazero}, we can easily complete the proof of the lemma:
indeed, by combining \eqref{eq:tociteIntro} with \eqref{eq:claimtlambdazero}, we have
\begin{align*}
 0 \leq \sup_{t\geq 0}\mathcal{J}_{\lambda,p}(t \eta_{\varepsilon})
 \leq g(t_{\e,\lambda})
 \leq \dfrac{t_{\e,\lambda}^2}{2}\big(\mathcal{S}_n+ C\varepsilon^{\kappa_{s,n}}\big) 
 - \dfrac{t_{\e,\lambda}^{2^*}}{2^{*}}\to 0\quad\text{as $\lambda\to \infty$},
\end{align*}
and this readily implies the existence of $\lambda_0 = \lambda_0(p,s,n,\e)> 0$ such that
$$\sup_{t\geq 0}\mathcal{J}_{\lambda,p}(t \eta_{\varepsilon}) < \frac{1}{n}(\mathcal{S}_n)^{2/n}\quad\text{for all
$\lambda\geq \lambda_0$},$$
provided that $\e > 0$ is small enough but \emph{fixed}. This ends the proof.
\end{proof}

\begin{proof}[Proof of Theorem~\ref{MS:3232}] The desired result in Theorem~\ref{MS:3232}
now follows from the Mountain Pass Theorem, thanks to Lemmata~\ref{prop:geometry}, \ref{prop:PalaisSmale} and~\ref{lem:Path}.
\end{proof}


\begin{thebibliography}{100}

\bibitem{AC21}
N. Abatangelo, M. Cozzi, {\em An elliptic boundary value problem with fractional nonlinearity},
SIAM J. Math. Anal., {\bf 53} (3) (2021), 3577--3601.

\bibitem{AmRa}
A. Ambrosetti, P. Rabinowitz, 
{\em Dual variational methods in critical point theory and applications}, J. Funct. Anal., {\bf 14} (1973), 349--381.

\bibitem{Aubin}
T. Aubin, 
{\em Probl\`emes isop\'erim\'etriques et espaces de Sobolev}, 
J. Differential Geometry, \textbf{11} (1976), 573--598. 

\bibitem{BCCI12} G. Barles, E. Chasseigne, A. Ciomaga, C. Imbert, {\em Lipschitz regularity of solutions for mixed integro-differential equations}, J. Differential Equations {\bf 252} (11) (2012) 6012--6060.

\bibitem{BI08} G. Barles, C. Imbert, {\em Second-order elliptic integro-differential equations:
viscosity solutions' theory revisited}, Ann. Inst. H. Poincar\'e C Anal. Non Lin\'eaire, {\bf
25} (3) (2008), 567--585.

\bibitem{BCSS}
B. Barrios, E. Colorado, R. Servadei, F. Soria, 
{\em A critical fractional equation with concave-convex power nonlinearities} Ann. Inst. H. Poincar\'{e} C Anal. Non Lin\'{e}aire {\bf 32}(4) (2015), 875--900.

\bibitem{BVDV21}
S. Biagi, S. Dipierro, E. Valdinoci, E. Vecchi, {\em
Semilinear elliptic equations involving mixed local and nonlocal operators},
Proc. Roy. Soc. Edinburgh Sect. A {\bf 151} (5) (2021), 1611--1641.

\bibitem{BDVV22a} S. Biagi, S. Dipierro, E. Valdinoci, E. Vecchi, {\em Mixed local
and nonlocal elliptic operators: regularity and maximum principles}, Comm. Partial
Differential Equations {\bf 47} (3) (2022), 585--629.

\bibitem{BDVV23}
S. Biagi, S. Dipierro, E. Valdinoci, E. Vecchi, {\em
A Hong-Krahn-Szeg\"o inequality for mixed local and nonlocal operators}, Math. Eng. {\bf 5} (1) (2023), Paper No. 014, 25 pp. 

\bibitem{BDVV}
S. Biagi, S. Dipierro, E. Valdinoci, E. Vecchi, {\em
A Faber-Krahn inequality for mixed local and nonlocal operators},
J. Anal. Math. \textbf{150} (2023), 405--448.

 \bibitem{BMV}
 S. Biagi, D. Mugnai, E. Vecchi,
{\em A Brezis-Oswald approach to mixed local and nonlocal operators}, 
Commun. Contemp. Math.  \textbf{26}(2) (2024).

\bibitem{BJK10} I. H. Biswas, E. R. Jakobsen, K. H. Karlsen, {\em Viscosity solutions for
a system of integro-PDEs and connections to optimal switching and control of jump-diffusion processes},
Appl. Math. Optim. {\bf 62} 1 (2010), 47--80.

\bibitem{BrLi}
H. Brezis, E. Lieb, 
{\em A relation between pointwise convergence of functions and convergence of functionals}, 
Proc. Am.Math. Soc. {\bf 88}(3), (1983), 486-490.

\bibitem{BN2}
H. Brezis, L. Nirenberg, 
{\em A minimization problem with critical exponent and nonzero data}, 
in Symmetry in Nature (a volume in honor of L. Radicati), Scuola Normale Superiore Pisa, 1989, Volume I, 129--140.
  
\bibitem{BrNir}
 H. Brezis, L. Nirenberg,
 {\em Positive solutions of nonlinear elliptic equations involving critical Sobolev exponents}, 
 Comm. Pure Appl. Math. \textbf{36} (1983), no. 4, 437--477.

\bibitem{CDV22} X. Cabr\'e, S. Dipierro, E. Valdinoci, {\em The Bernstein technique for integro-differential equations},
Arch. Ration. Mech. Anal., {\bf 243} (3) (2022), 1597--1652.

\bibitem{CKSV12} Z.-Q. Chen, P. Kim, R. Song, Z. Vondra\v{c}ek, {\em Boundary {H}arnack principle for {$\Delta+\Delta^{\alpha/2}$}},
Trans. Amer. Math. Soc. {\bf 364} (8) (2012), 4169--4205.

\bibitem{DeFMin}
C. De Filippis, G. Mingione,
\emph{Gradient regularity in mixed local and nonlocal problems}, 
Math. Ann. \textbf{388} (2024), 261--328.


\bibitem{DRV}
E. Di Nezza, G. Palatucci, E. Valdinoci,
{\em Hitchhiker's guide to the fractional Sobolev spaces}, 
Bull. Sci. Math. {\bf 136} (2012), 521--573.

\bibitem{DLPV} S. Dipierro, E. Proietti Lippi, E. Valdinoci, {\em (Non)local logistic equations with Neumann conditions},
Ann. Inst. H. Poincar\'e Anal. Non Lin\'{e}aire {\bf 40}(5), (2023), no. 5, 1093--1166.


\bibitem{DV21} S. Dipierro, E. Valdinoci, {\em Description of an ecological niche for a mixed
local/nonlocal dispersal: an evolution equation and a new Neumann condition arising from the superposition of Brownian and L\'evy processes}, Phys. A {\bf 575} (2021) Paper No.
126052, 20.

\bibitem{GaKi}
P. Garain, J. Kinnunen, 
{\em On the regularity theory for mixed local and nonlocal quasilinear elliptic equations},
Trans. Amer. Math. Soc. \textbf{375} (2022), 5393--5423.
 

\bibitem{GaLi}
P. Garain, E. Lindgren, 
{\em Higher Hölder regularity for mixed local and nonlocal degenerate elliptic equations}, 
Calc. Var. Partial Differential Equations \textbf{62} (2023).

\bibitem{GaUk}
P. Garain, A. Ukhlov,
{\em Mixed local and nonlocal Sobolev inequalities with extremal and associated quasilinear singular elliptic problems},
Nonlinear Anal. {\bf 223}, (2022), 113022.

\bibitem{JK05} E. R. Jakobsen, K. H. Karlsen, {\em Continuous dependence estimates for viscosity solutions of integro-PDEs}, J. Differential Equations, {\bf 212} (2) (2005), 278--318.

\bibitem{MPV13}
E. Montefusco, B. Pellacci, G. Verzini, {\em Fractional diffusion
with Neumann boundary conditions: the logistic equation}, Discrete Contin. Dyn. Syst.
Ser. B, {\bf 18} (8) (2013), 2175--2202.

\bibitem{MPSY16}
S. Mosconi, K. Perera, M. Squassina, Y. Yang, {\em The Brezis-Nirenberg problem for the fractional $p$-Laplacian}, Calc. Var. Partial Differential Equations {\bf 55} (4) (2016), Art. 105, 25 pp.

\bibitem{PV18} B. Pellacci, G. Verzini, {\em Best dispersal strategies in spatially heterogeneous environments: optimization of the principal eigenvalue for indefinite fractional
Neumann problems}, J. Math. Biol. {\bf 76} (6) (2018), 1357--1386.

\bibitem{RSPoho}
X. Ros-Oton,  J. Serra,
{\em Nonexistence Results for Nonlocal Equations with Critical and Supercritical Nonlinearities},
Comm. Partial Differential Equations \textbf{40} (2015), no. 1, 115--133.

\bibitem{SV}
A. Salort, E. Vecchi, {\em On the mixed local-nonlocal H\'enon equation},
Differential Integral Equations {\bf 35}(11-12), (2022).

 \bibitem{SerValEigen}
 R. Servadei, E. Valdinoci, 
 \emph{Variational methods for non-local operators of elliptic type}, 
 Discrete Contin. Dyn. Syst. \textbf{33} (2013), 2105--2137.
 
 \bibitem{SerValTAMS}
R. Servadei, E. Valdinoci,
{\em The Brezis-Nirenberg result for the fractional Laplacian}, Trans. Amer. Math. Soc. {\bf 367}(1), (2015),  67--102.

\bibitem{SUPR} X. Su, E. Valdinoci, Y. Wei,
{\em Regularity results for solutions of mixed local and nonlocal elliptic equations},
Math. Z. \textbf{302} (2022), 1855--1878.

\bibitem{Talenti}
 G. Talenti, 
 {\em Best constant in Sobolev inequality},  
 Ann. Mat. Pura Appl. \textbf{110} (1976), 353--372.

\end{thebibliography}
\end{document}